\documentclass[a4paper, 11pt]{article}
\usepackage[utf8]{inputenc}
\usepackage[headings]{fullpage}               
\usepackage{boxedminipage}
\usepackage{amsfonts}
\usepackage{amsmath} 
\usepackage{amssymb}
\usepackage{graphicx}
\usepackage{amsthm}
\usepackage{tikz}
\usepackage{setspace}
\usepackage[colorlinks=true,linkcolor=blue,urlcolor=blue,citecolor=red]{hyperref}
\usepackage{enumerate}
\usepackage{algorithm}
\usepackage{algpseudocode}
\usepackage[title]{appendix}

\usepackage{verbatim}

\usepackage{kbordermatrix}
\hypersetup{linktocpage=true,}
\newtheorem{theorem}{Theorem}[section]
\newtheorem{lemma}[theorem]{Lemma}
\newtheorem{claim}[theorem]{Claim}

\newtheorem{proposition}[theorem]{Proposition}
\newtheorem{corollary}[theorem]{Corollary}
\theoremstyle{definition}
\theoremstyle{definition}\newtheorem{example}[theorem]{Example}
\theoremstyle{definition}\newtheorem{remark}[theorem]{Remark}

\newcommand{\rank}{\operatorname{rank}}

\newcommand{\sgn}{\operatorname{sgn}}

\tikzstyle{vertex}=[circle, draw, fill=black, inner sep=0pt, minimum size=4pt]
\tikzstyle{edge}=[line width=1.5pt]

\title{Rigid graphs in cylindrical normed spaces}
\author{Sean Dewar\thanks{School of Mathematics, University of Bristol, Bristol, UK. E-mail: \texttt{sean.dewar@bristol.ac.at}} \and Derek Kitson\thanks{Department of Mathematics and Computer Studies, Mary Immaculate College, Thurles,
Ireland. E-mail: \texttt{derek.kitson@mic.ul.ie}}}
\date{May 15, 2023}

\begin{document}

\maketitle

\begin{abstract}
   We characterise rigid graphs for cylindrical normed spaces $Z=X\oplus_\infty \mathbb{R}$ where $X$ is a finite dimensional real normed linear space and $Z$ is endowed with the product norm.
   In particular, we obtain purely combinatorial characterisations of minimal rigidity for a large class of 3-dimensional cylindrical normed spaces;
   for example, when $X$ is an $\ell_p$-plane with $p\in (1,\infty)$.  
   We combine these results with recent work of Cros et al.~to characterise rigid graphs in the $4$-dimensional cylindrical space $(\mathbb{R}^2\oplus_1\mathbb{R})\oplus_\infty\mathbb{R}$.
    These are among the first combinatorial characterisations of rigid graphs in normed spaces of dimension greater than 2.
   Examples of rigid graphs are presented and algorithmic aspects are discussed. 
\end{abstract}

\section{Introduction}
A simple undirected graph $G=(V,E)$ is {\em flexible} in a real normed linear space $Z$ if given any placement $p=(p_v)_{v\in V}\in Z^V$ of the vertices in $Z$ there exists a non-trivial continuous motion of the {\em joints} $p_v$ ($v\in V$) which preserves the lengths of each of the {\em bars} $p_v-p_w$ ($vw\in E$). 
A graph which is not flexible is said to be {\em rigid} in $Z$.
If $\dim Z=1$ then it is well known that the rigid graphs are precisely the connected graphs. If $\dim Z=2$ then there is a dichotomy; either $Z$ is isometrically isomorphic to the Euclidean plane, in which case the rigid graphs have been characterised by Pollaczek-Geiringer \cite{pollaczek27}, or $Z$ is not isometrically isomorphic to the Euclidean plane, in which case the rigid graphs are characterised in \cite{dew2}. 
There are currently no known combinatorial characterisations for rigid graphs in Euclidean spaces of dimension $3$ or higher.

The rigidity of graphs in non-Euclidean normed spaces of dimension $d\geq 3$ has previously been investigated in the settings of $\ell_p$-spaces \cite{DewKitNix,kit-pow}, polyhedral spaces \cite{polyhedra}, matrix spaces \cite{matrixnorm}, and for a class of mixed norms \cite{CruKasKitSch}. In each of these settings, necessary combinatorial conditions for rigidity have been derived but complete characterisations have remained elusive.  In this article we present complete combinatorial characterisations of rigid graphs for classes of 3-dimensional normed spaces, namely the cylindrical spaces $Z=X\oplus_\infty \mathbb{R}$ where $X$ is a ``generic" normed plane and $Z$ is endowed with the product norm $\|(x,y)\|_Z=\max\{\|x\|_X,|y|\}$. 
The class of generic normed planes is defined in Section \ref{Sec:Rigidity} and includes all $\ell_p$-planes where $p\in(1,\infty)$.  Cylindrical normed spaces are a special case of the more general product normed spaces  considered in \cite{matrixnorm}. In the specific case of the cylindrical normed space $\mathbb{R}^2\oplus_\infty \mathbb{R}$, our results settle  Conjecture 59 from \cite{matrixnorm} in the affirmative. We highlight also the recent work of Cros et al.~(\cite{CAPR}) on a related notion of rigidity motivated by formation control. Although not explicitly stated, their work indicates a characterisation of rigidity for the conical space $\mathbb{R}^2\oplus_1 \mathbb{R}$ which we present in Section \ref{s:conical}. 

In Section \ref{Sec:Frameworks}, we introduce bar-joint frameworks $(G,p)$ in the setting of a cylindrical normed space $Z=X\oplus_\infty \mathbb{R}$. In preparation for later sections, we establish equivalent conditions for a framework to be well-positioned in $Z$ and we describe the two {\em monochrome subframeworks} of $(G,p)$ which are induced by the cylindrical norm on $Z$.
We also prove a key lemma which establishes the existence of special placements for graphs in cylindrical normed spaces in which a given spanning forest is realised as one of the two induced monochrome subframeworks.
The main results are contained in Section \ref{Sec:Rigidity}. Firstly, in Theorem \ref{graphthm} we characterise minimal rigidity for cylindrical normed spaces of dimension at least 3 in terms of a packing consisting of two spanning subgraphs which are, respectively, minimally rigid in $X$ and a tree. In the special case where $X$ is a generic normed plane we obtain complete combinatorial characterisations of minimal rigidity (Theorem \ref{mainthm}).
In Section \ref{S:Examples}, we draw on surface triangulations to construct examples of rigid graphs and we identify connectivity criteria which guarantee rigidity in cylindrical normed spaces.
Finally, in Section \ref{S:Related}, we present several related results and highlight some open problems. In particular, we establish an equivalence between rigidity in the cylindrical space $\mathbb{R}^2\oplus_\infty \mathbb{R}$ and rigidity in the dual space $\mathbb{R}^2\oplus_1 \mathbb{R}$, and we characterise rigid graphs for the 4-dimensional cylindrical space $(\mathbb{R}^2\oplus_1\mathbb{R})\oplus_\infty\mathbb{R}$.

\section{Frameworks in cylindrical normed spaces}
\label{Sec:Frameworks}
Throughout, $X$ denotes a finite dimensional real normed linear space with norm $\|\cdot\|_X$ and $\mathbb{R}^d$ denotes $d$-dimensional Euclidean space with the standard $\ell_2$ norm $\|\cdot\|_2$ (or $|\cdot|$ if $d=1$). For $p\in[1,\infty]$ and $d\geq 2$, we denote by $\ell_p^d$ the $d$-dimensional $\ell_p$-space (note that $\ell_2^d=\mathbb{R}^d$).
All graphs are finite, simple and undirected. Given a graph $G$ with a subgraph $H$, we denote by $G-H$ the subgraph of $G$ formed by removing every edge of $H$ from $G$.
 A \emph{(bar-joint) framework} in  $X$ is a pair  $(G,p)$  consisting of a graph $G=(V,E)$ and a vector $p\in X^V$, $p=(p_v)_{v\in V}$,  with the property that $p_v\not= p_w$ for all edges $vw\in E$.
 The vector $p$ is referred to as a \emph{placement} of $G$ in $X$.
A \emph{subframework} of $(G,p)$ is a  framework $(H,p_H)$ where $H=(V(H),E(H))$ is a subgraph of $G$ and $p_H=(p_v)_{v\in V(H)}\in X^{V(H)}$ is the induced placement of $H$ in $X$.

\subsection{Cylindrical normed spaces}
A normed linear space is said to be {\em cylindrical} if it is isometrically isomorphic to a direct sum $Z =X\oplus_\infty \mathbb{R}$ where $X$ is a normed linear space and $Z$ is endowed with the product norm $\|(x,y)\|_Z := \max\{ \|x\|_X,|y|\}$. Note that the closed unit ball of $Z$ is the Cartesian product $B_X\times [-1,1]$ where $B_X$ denotes the closed unit ball of $X$.
We denote by $\pi_X$ and $\pi_{\mathbb{R}}$ the natural projections,
\begin{equation}\label{eq:natproj}
    \pi_X :Z \rightarrow X, ~ (x,y) \mapsto x, \qquad \pi_{\mathbb{R}} :Z \rightarrow \mathbb{R}, ~ (x,y) \mapsto y.
\end{equation}
Let $K$ denote the double cone $K=\{(x,y)\in Z:\|x\|_X \leq |y|\}$. The interior, boundary and complement of $K$ are denoted respectively by $K^\circ$, $\partial K$ and $K^c$. 

\begin{example}
\label{ex:matrixnorms}
Setting $X=\ell_q^d$, with $q\in[1,\infty)$,  we obtain the cylindrical normed space $\ell_{q,\infty}^{d+1}:=\ell_q^d\oplus_\infty \mathbb{R}$ with norm,
\[\|(x_1,\ldots,x_d,x_{d+1})\|_{q,\infty} = \max\left\{\left(\sum_{i=1}^d |x_i|^q\right)^{\frac{1}{q}},\,|x_{d+1}|\right\}.\]
See Figure \ref{fig:unitball} (left) for an illustration of the unit ball and double cone in $\ell_{2,\infty}^3 = \mathbb{R}^2 \oplus_\infty \mathbb{R}$.

\end{example}

\begin{figure}[ht]
\begin{minipage}{0.4\linewidth}
    \centering
    \includegraphics[width=10cm, height=4.2cm]{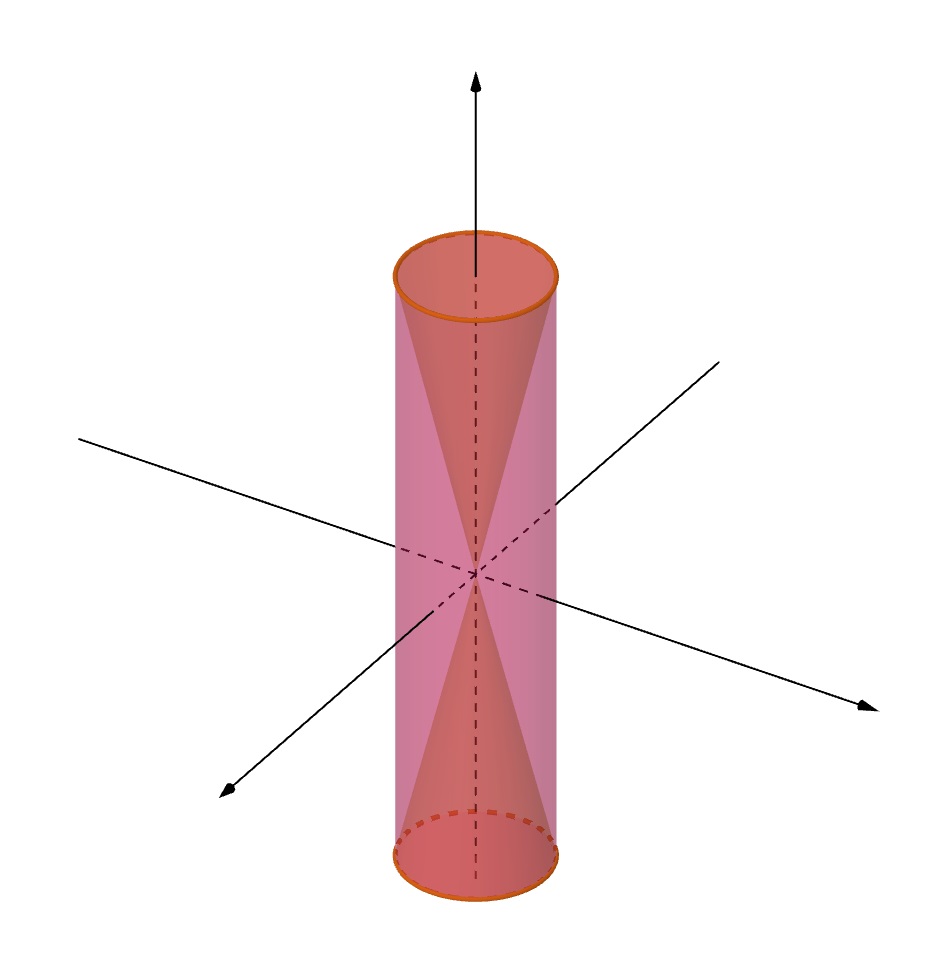}
    \end{minipage}
    \begin{minipage}{0.45\linewidth}
    \centering
     \includegraphics[width=10cm, height=3.8cm,trim={7cm 0 4cm 3.8cm},clip]{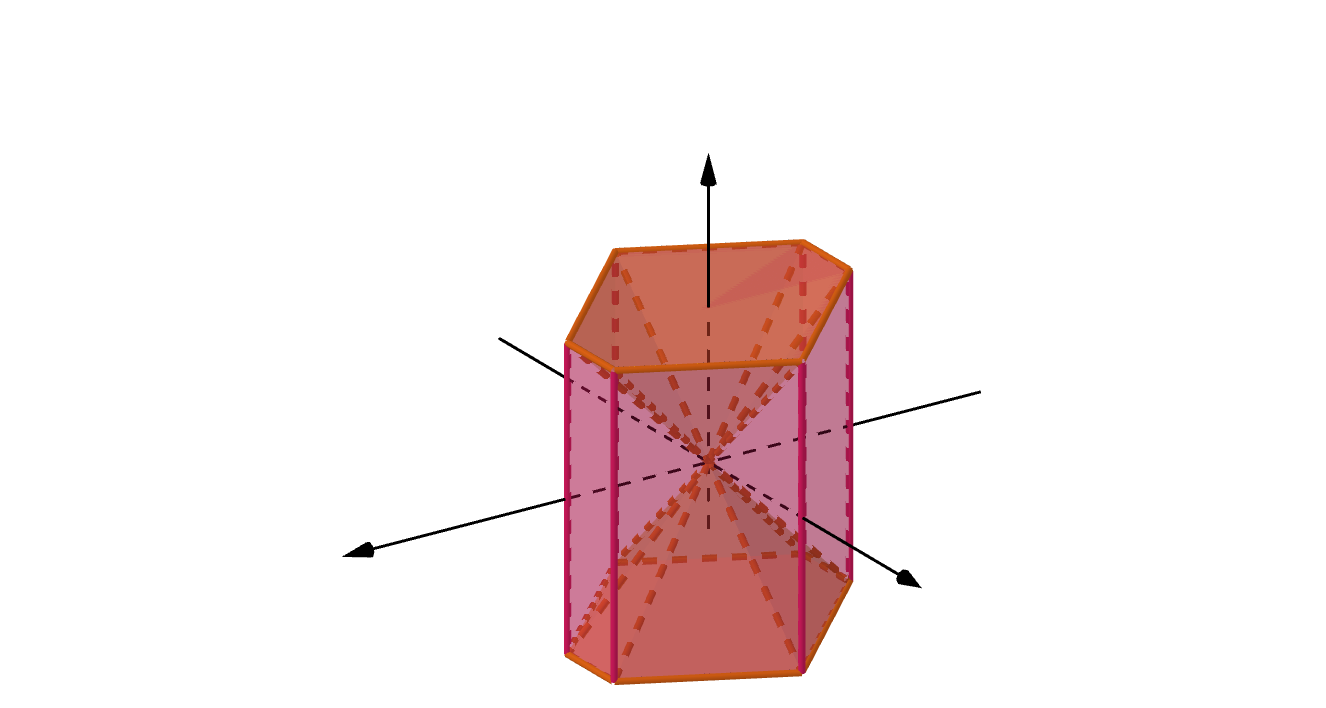}    
    \end{minipage}
    \caption{Left: An illustration of the unit ball for the cylindrical normed space $\ell_{2,\infty}^3 = \mathbb{R}^2 \oplus_\infty \mathbb{R}$. 
    Right: An illustration of the unit ball for a cylindrical normed space $X \oplus_\infty \mathbb{R}$ where $X$ has a polygonal unit ball. 
    In both pictures, the darker regions indicate the intersection of the unit ball with the double cone $K$.}
    \label{fig:unitball}
\end{figure}

\begin{example}
Let $X$ be a linear space endowed with a polyhedral norm (i.e.~the unit ball $B_X$ is a convex polyhedron). The unit ball for the cylindrical normed space $Z=X \oplus_\infty \mathbb{R}$ is the prism $B_Z=B_X\times [-1,1]$ and so  $\|\cdot\|_Z$ is a polyhedral norm on $Z$.
For example, setting $X=\ell_\infty^d$  we obtain the cylindrical normed space $\ell_{\infty,\infty}^{d+1}:=\ell_\infty^d\oplus_\infty \mathbb{R}=\ell_\infty^{d+1}$.
See Figure \ref{fig:unitball} (right)  for an illustration of the unit ball in a $3$-dimensional cylindrical normed space $X \oplus_\infty \mathbb{R}$ where $X$ is a $2$-dimensional space with an irregular hexagonal unit ball.

\end{example}

\begin{example}
Let $\mathcal{H}_1(2,\mathbb{R})$ denote the real linear space of $2\times 2$ real symmetric matrices endowed with the trace norm and let $\mathcal{H}_1(2,\mathbb{C})$ denote the real linear space of $2\times 2$ complex hermitian  matrices also endowed with the trace norm.
As noted in \cite{matrixnorm}, $\mathcal{H}_1(2,\mathbb{R})$ is isometrically isomorphic to  $\ell_{2,\infty}^3 = \mathbb{R}^2 \oplus_\infty \mathbb{R}$ and $\mathcal{H}_1(2,\mathbb{C})$ is isometrically isomorphic to  $\ell_{2,\infty}^4 = \mathbb{R}^3 \oplus_\infty \mathbb{R}$.
Hence $\mathcal{H}_1(2,\mathbb{R})$ and $\mathcal{H}_1(2,\mathbb{C})$ are examples of cylindrical normed spaces.
\end{example}

\subsection{Key Lemma}
\label{Sec:Key}
Let $x\in X$ and let $r>0$.
In the following, we let $S_r[x]=\{y\in X:\|y-x\|_X=r\}$ denote the sphere in $X$ of radius $r$ centred at  $x$. We denote by $B_r[x]=\{y\in X:\|y-x\|_X\leq r\}$ the closed ball in $X$ of radius $r$ centred at $x$.

\begin{lemma}\label{mainlem}
    Let $X$ be a normed space with dimension at least $2$ and let $G=(V,E)$ be a graph with a spanning forest $T$.
    Then there exists a  framework $(G,p)$ in the cylindrical normed space $Z=X \oplus_\infty \mathbb{R}$ where:
    \begin{enumerate}[(i)]
        \item $p_v-p_w\in K^\circ$ for every edge $vw \in T$, and
        \item $p_v-p_w\in K^c$ for every edge $vw \in E \setminus T$.
    \end{enumerate}
\end{lemma}

\begin{proof}
    First suppose $T$ is a tree.
    Choose a root $v_0$ of the tree $T$ and define $d_T$ to be the metric on the set $V$ where $d_T(v,w)$ is the distance between vertices $v$ and $w$ in $T$.
    Given $n = \max_{v \in V} d_T(v,v_0)$,
    define the sets $V_0,\ldots,V_n$ where $V_i := \{ v \in V : d_T(v,v_0) = i\}$.
    We will now construct the following: (i) a set of distinct points $\{q_v :v \in V \}$ in $X$, and (ii) strictly decreasing positive scalars $s_0>s_1>\cdots>s_n$.
    First,
    set $q_{v_0}=0$ and $s_0 = 1$.
    Next,
    for each $v \in V_1$ choose a point $q_v$ on the unit sphere of $X$ so that no two vertices are mapped to the same point.
    Now choose $s_1$ such that $0< s_1 < \frac{1}{2}\min \left\{ \|q_v - q_w\|_X : v, w \in V_0\cup V_1, ~ v\neq w\right\}$. Note in particular that $s_1<\frac{1}{2}s_0$.
    
    We will now define the scalars $s_0,\ldots,s_n$ and the set of distinct points $\{q_v :v \in V \}$ with the following inductive algorithm.
    \begin{enumerate}[1)]
        \item First assume that for some $2\leq k\leq n$,
    we have chosen $s_0,\ldots,s_{k-1}$ and $q_v$ for all $v \in \bigcup_{i=0}^{k-1}V_i$.
        \item For each $u \in V_k$,
        there exists a unique path $(u,u_1,u_2)$ in $T$ with $u_1 \in V_{k-1}$ and $u_2 \in V_{k-2}$.
        \item Choose a set $\{q_u : u \in V_k\}$ of distinct points so that $q_u \in S_{s_{k-1}} [q_{u_1}] \setminus B_{s_{k-2}} [q_{u_2}]$ for each $u \in V_k$.
        \item Pick $0< s_k < \frac{1}{2}s_{k-1}$ sufficiently small so that any two spheres $S_{s_i}[q_v]$ and $S_{s_j}[q_w]$, where $v\in V_i$, $w \in V_j$ and $i,j\in\{0,\ldots, k\}$, intersect if and only if $v$ and $w$ are adjacent in $T$. 
        Specifically, choose $s_k$ so that $s_k < \frac{1}{2} \min \left\{ \|q_v - q_w\|_X : v, w \in V_k, ~ v\neq w\right\}$ and $s_k < \|q_u-q_v\|_X -s_i$ for all $u \in V_k$ and $v \in V_i \setminus \{u_1\}$ for $i < k$.
    \end{enumerate}
    Note that the set of distinct points $\{q_v :v \in V \}$ and the scalars $s_0,\ldots,s_{n}$ will have the following property:
    for any vertices $v ,w\in V$ with $v \in V_i$ and $w \in V_j$ and $i \leq j$,
    we have $\|q_v - q_w\|_X = s_i$ if $v$ and $w$ are adjacent in $T$,
    and $\|q_v - q_w\|_X > s_i$ otherwise.
    
    
    Now choose some small $\epsilon >0$ so that the following holds.
    The strictly decreasing sequence $r_0 > \ldots > r_n >0$ with $r_i = s_i + \epsilon/2^i$ has the property that for any vertices $v \in V_i$ and $w \in V_j$ with $j \geq i$,
    we have $\|q_v - q_w\|_X < r_i$ if $v$ and $w$ are adjacent in $T$,
    and $\|q_v - q_w\|_X > r_i$ otherwise.
    We note that we will also have that $r_i < \frac{1}{2}r_{i-1}$ for every $i \in \{1,\ldots,n\}$.
    With this,
    define for each $1 \leq k \leq n$ the value $h_k := \sum_{i=0}^{k-1}(-1)^i r_i$.
    We now define $p$ to be the placement in $X \oplus_\infty \mathbb{R}$ where $p_{v_0}=0$ and for each vertex $v \in V_k$ we have $p_v = ( q_v , h_k)$.
    
    \begin{claim}
        Let $0 \leq i \leq j \leq n$.
        Then for any vertices $v ,w\in V$ with $v \in V_i$ and $w \in V_j$,
        $\|q_v - q_w\|_X < |h_i - h_j|$ if $v$ and $w$ are adjacent in $T$ and $\|q_v - q_w\|_X > |h_i - h_j|$ otherwise.
    \end{claim}
    
    \begin{proof}
        If $i=j$ then $\|q_v - q_w\|_X > |h_i - h_i| =0$.
        If $j = i+1$ then $|h_i - h_j| = r_i$.
        By our choice of $r_i$,
        we have $\|q_v - q_w\|_X < r_i$ if $v$ and $w$ are adjacent in $T$,
        and $\|q_v - q_w\|_X > r_i$ otherwise.
        Finally,
        suppose $j > i+ 1$.
        As $r_k < r_{k-1}/2$ for each $k = 1,\ldots,n$,
        we have $|h_i - h_j| = \sum_{k=i}^{j-1} (-1)^{k-i} r_k < r_i$.
        Since $v$ and $w$ cannot be adjacent in $T$ then by our choice of $r_i$ we have $\|q_v - q_w\|_X > r_i > |h_i - h_j|$.
    \end{proof}
    
    It now follows that $p_v-p_w\in K^\circ$ for every edge $vw \in T$ and
         $p_v-p_w\in K^c$ for every edge $vw \in E \setminus T$. 
    
    Now suppose $T$ is a forest with connected components $T_1,\ldots,T_m$.
    Apply the same process to each subgraph $G_i$ induced on the vertex set $V_i$ of $T_i$ to obtain a placement $p^i$ in $X \oplus_\infty \mathbb{R}$.
    Set 
    \begin{align*}
        R = \max \left\{ \left|\pi_{\mathbb{R}} (p^i_v -p^j_w) \right| : v \in V_i, w \in V_j  \right\}, \qquad S = \max \left\{ \left\|\pi_{X} (p^i_v-p^j_w) \right\|_X : v \in V_i, w \in V_j  \right\}
    \end{align*}
    and choose any vector $x \in X$ with $\|x\|_X > R+S$.
    We now set $p$ to be the placement of $G$ in $X \oplus_\infty \mathbb{R}$ with $p_v = p^i_v + (ix,0)$ for each $v \in V_i$.
    We note that for $i \neq j$ and vertices $v \in V_i$ and $w \in V_j$,
    we will have 
    \begin{align*}
        \left\| \pi_X(p_v-p_w)\right\|_X \geq \left| i-j \right| \|x\|_X - \left\|\pi_X(p^i_v-p^j_w)\right\|_X  > R \geq \left| \pi_{\mathbb{R}}(p_v-p_w)\right|. 
    \end{align*}
    Thus, $p_v-p_w\in K^\circ$ for every edge $vw \in T$ and
         $p_v-p_w\in K^c$ for every edge $vw \in E \setminus T$.
\end{proof}

\begin{remark}
The proof of Lemma \ref{mainlem} requires that $\dim X\geq 2$. If $\dim X=1$ then  the cylindrical space $X\oplus_\infty \mathbb{R}$ is isometrically isomorphic to the $\ell_\infty$-plane. In this setting a result somewhat analogous to Lemma \ref{mainlem} was obtained in \cite[Theorem 4.3]{clikit} using different methods: if a graph $G$ can be expressed as an edge disjoint union of two spanning trees $T_1$ and $T_2$ then there exists a framework $(G,p)$ in the $\ell_\infty$-plane such that $p_v-p_w\in K^\circ$ for all $vw\in T_1$ and $p_v-p_w\in K^c$ for all $vw\in T_2$.
The proof of \cite[Theorem 4.3]{clikit} uses a multigraph inductive construction with accompanying geometric arguments to construct the placement 
$p$. This approach does not appear to adapt easily to higher dimensional settings.     
\end{remark}

\subsection{Monochrome subframeworks}
Let $(G,p)$ be a framework in a cylindrical normed space $X\oplus_\infty \mathbb{R}$ with the property that $p_v-p_w\notin \partial K$ for each edge $vw\in E$. 
Define an edge-labelling $\kappa_p: E\to
\{\mbox{blue},\,\mbox{green}\}$ where for each edge $vw\in E$, $$\kappa_p(vw) = \left\{\begin{array}{ll}
\mbox{blue} & \mbox{ if }p_v-p_w\in K^c\\
\mbox{green} & \mbox{ if }p_v-p_w\in K^\circ.
\end{array}\right.$$
Let $G_{X,p}$ denote the subgraph of $G=(V,E)$ consisting of the vertex set $V$ and all blue edges. Similarly, let $G_{\mathbb{R},p}$ denote the subgraph of $G$ consisting of the vertex set $V$ and all green edges.
When the context is unambiguous,
we denote these graphs by $G_X$ and $G_\mathbb{R}$ respectively.
The subframeworks $(G_X,p)$ and $(G_{\mathbb{R}},p)$ are referred to as the  {\em monochrome subframeworks} of $(G,p)$.

Define $(G_X,p_X)$ to be the framework in the normed space $X$ with graph $G_X$ and placement $p_X := ( \pi_X(p_v))_{v \in V}$. 
Similarly, define $(G_\mathbb{R}, p_\mathbb{R})$ to be the framework in $\mathbb{R}$ with graph $G_{\mathbb{R}}$ and placement $p_\mathbb{R} := ( \pi_\mathbb{R}(p_v))_{v \in V}$.
We refer to $(G_X,p_X)$ and $(G_\mathbb{R}, p_\mathbb{R})$ as the {\em projected monochrome subframeworks} of $(G,p)$.

\begin{example}
\label{ex:mono}
Consider the framework $(K_3,p)$ in $\mathbb{R}^2\oplus_\infty \mathbb{R}$,
where $K_3$ is the complete graph with vertex set $V=\{v_1,v_2,v_3\}$ and $p$ is the placement 
$p_{v_1} = (1,-1, \frac{1}{2})$, $p_{v_2} = (1,1,1)$ and $p_{v_3}=(\frac{3}{2},-1,\frac{3}{2})$.
Note that
\begin{gather*}
    p_{v_1}-p_{v_2}=\left(0,-2,-\frac{1}{2} \right)\in K^c, \quad p_{v_2}-p_{v_3}=\left(-\frac{1}{2},2,-\frac{1}{2} \right)\in K^c, \\
 p_{v_1}-p_{v_3}=\left(-\frac{1}{2},0,-1 \right)\in K^\circ.
\end{gather*}
With respect to the induced edge-labelling $\kappa_p$, the edges $v_1v_2$ and $v_2v_3$ are coloured blue and the remaining edge $v_1v_3$ is coloured green. See Figure \ref{fig:monochrome} for an illustration of this framework together with the induced edge-labelling and projected monochrome subframeworks.
\end{example}

\begin{figure}[ht]
    \centering
    \includegraphics[width=14cm, height=6cm]{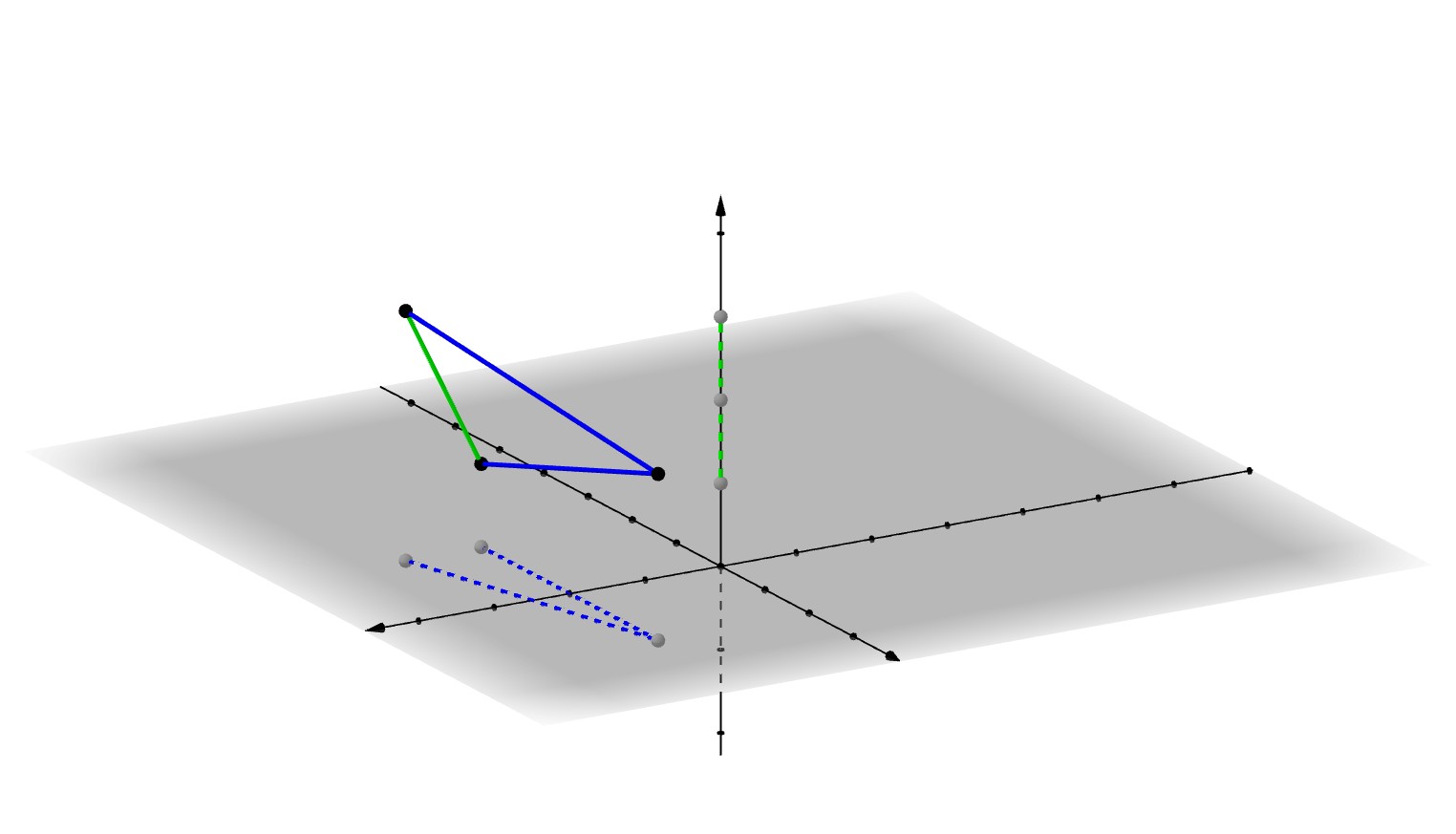}
    \caption{An illustration of the framework $(K_3,p)$ in $\mathbb{R}^2\oplus_\infty\mathbb{R}$ described in Example \ref{ex:mono} together with the induced edge-labelling. The projected monochrome subframeworks are illustrated with dashed lines.  }
    \label{fig:monochrome}
\end{figure}

\subsection{Smooth points and support functionals}
Denote by $X^*$ the dual space of the normed space $X$ (i.e., the set of all real-valued linear functionals on $X$).
A {\em support functional} for a point $x\in X$ is a linear functional $\varphi\in X^*$ which satisfies $\varphi(x)=\|x\|_X^2$ and $\|\varphi\|_{X^*} := \sup_{\|u\|=1} |\varphi(u)| =\|x\|_X$.
A non-zero point $x\in X$ is said to be \emph{smooth} if it has a unique support functional.
We denote by $\varphi_x$ the unique support functional for a smooth point $x$.
Equivalently, a non-zero point $x\in X$ is smooth if the map $\rho_u:\mathbb{R}\to\mathbb{R}$, $t \mapsto \|x+tu\|_X$ is differentiable at $t=0$ for each $u\in X$. In this case, the unique support functional for $x$ satisfies $\varphi_x(u)=\rho_u'(0)\|x\|_X$ for each $u\in X$; see \cite[Lemma 1]{kit-sch} for more details.

\begin{lemma}
\label{l:diff}
Let $Z=X\oplus_\infty \mathbb{R}$  be a cylindrical normed space and let $z=(x,y)\in Z$ be a non-zero point. 
\begin{enumerate}[(i)]
\item If $z \in K^\circ$ then $z$ is a smooth point in $Z$ with unique support functional,
\[\varphi_z: Z\to \mathbb{R}, \quad \varphi_z(a,b) = yb.\]
\item If $z\in\partial K$ then $z$ is not a smooth point in $Z$.
\item If $z\in K^c$ then $z$ is a smooth point in $Z$ if and only if $x$ is a smooth point in $X$. Moreover, if $z$ is smooth then its unique support functional is
\[\varphi_z:Z\to \mathbb{R}, \quad \varphi_z(a,b) = \varphi_x(a).\]
\end{enumerate}
\end{lemma}

\proof


$(i)$ If $z\in K^\circ$ then, for each  $u=(a,b)\in Z$, it suffices to note that
$\|z+tu\|_Z = |y+tb|$ for $|t|$ sufficiently small.

$(ii)$
Suppose $z \in \partial K \setminus \{0\}$.
As $y \neq 0$,
the point $u=(0,\sgn(y))\in Z$ is well-defined.
We now note that
\[\lim_{t\to0^+} \frac{1}{t}(\|z+tu\|_Z-\|z\|_Z) 
= \lim_{t\to0^+} \frac{1}{t}(|y+t\sgn(y)|-|y|)
= 1,\] whereas,
\[\lim_{t\to0^-} \frac{1}{t}(\|z+tu\|_Z-\|z\|_Z) 
= \lim_{t\to0^-} \frac{1}{t}(\|x\|_X-\|x\|_X)
= 0.\]

$(iii)$ If $z\in K^c$ then, for each  $u=(a,b)\in Z$, it suffices to note that
$\|z+tu\|_Z
= \|x+ta\|_X$
for $|t|$ sufficiently small.
\endproof

\begin{example}
\label{ex:cyl1}
Consider the cylindrical normed space  $\ell_{q,\infty}^{d+1}=\ell_q^d\oplus_\infty \mathbb{R}$ where $q\in[1,\infty)$ and $d\geq 1$. Let $z=(x,y)\in \ell_q^d\oplus_\infty \mathbb{R}$ be a smooth point  with unique support functional $\varphi_z$ and write $x=(x_1,\ldots, x_d)$. 
Let $u=(a,b)\in\ell_q^d\oplus_\infty \mathbb{R}$ and write $a=(a_1,\ldots,a_d)$.
If  $z\in K^\circ$ then, by Lemma \ref{l:diff}(i),
\[\varphi_z(u) = yb = \begin{bmatrix}0&\cdots&0&y\end{bmatrix}\begin{bmatrix}a_1\\\vdots\\a_d\\b\end{bmatrix}\]
If $z\in K^c$ then, by Lemma \ref{l:diff}(iii),
\[
\varphi_z(u) = \varphi_x(a) = \sum_{i=1}^d \frac{\sgn(x_i)|x_i|^{q-1}}{\|x\|_q^{q-2}}a_i= \begin{bmatrix}\frac{\sgn(x_1)|x_1|^{q-1}}{\|x\|_q^{q-2}}&\cdots&\frac{\sgn(x_d)|x_d|^{q-1}}{\|x\|_q^{q-2}}&0\end{bmatrix}\begin{bmatrix}a_1\\\vdots\\a_d\\b\end{bmatrix}\]
The cylindrical normed space  $\ell_{2,\infty}^3=\mathbb{R}^2\oplus_\infty \mathbb{R}$ is of particular interest. Here, if $z\in K^\circ$ then,
$\varphi_z(u) = \begin{bmatrix}0&0&y\end{bmatrix}\begin{bmatrix}a_1\\a_2\\b\end{bmatrix}$
and if  $z\in K^c$ then,
$
\varphi_z(u) = \begin{bmatrix}x_1 &x_2&0\end{bmatrix}\begin{bmatrix}a_1\\a_2\\b\end{bmatrix}$.
\end{example}




\subsection{Well-positioned frameworks}
\label{s:wp}
 The  \emph{rigidity map} for a normed space $X$ and graph $G=(V,E)$ is,
\begin{align*}
    f_G : X^V \rightarrow \mathbb{R}^E, \quad (x_v)_{v \in V} \mapsto \left( \|x_v-x_w\|_X  \right)_{vw\in E}.
\end{align*}
A framework $(G,p)$ in $X$ is {\em well-positioned} if the rigidity map $f_G$ is differentiable at $p$. 
 If $(G,p)$ is well-positioned and the differential $df_G(\cdot)$ achieves its maximum rank at $p$,
 then the framework $(G,p)$ is said to be {\em regular} and $p$ is said to be a {\em regular placement} of $G$ in $X$. 
Furthermore, if the differential $df_G(p):X^V\to \mathbb{R}^E$ is surjective then the framework $(G,p)$ is said to be {\em independent}. 

We will require the following two lemmas.


\begin{lemma}\label{lem:wp}
Let $(G,p)$ be a framework in a cylindrical normed space $X\oplus_\infty \mathbb{R}$.
The following statements are equivalent.
\begin{enumerate}[(i)]
    \item $(G,p)$ is well-positioned in $X\oplus_\infty \mathbb{R}$.
    \item  $p_v-p_w$ is a smooth point in $X\oplus_\infty \mathbb{R}$ for every edge $vw\in E$.
    \item $p_v-p_w\notin \partial K$ for each edge $vw\in E$ and the projected monochrome subframework $(G_X,p_X)$ is well-positioned in $X$.
\end{enumerate}
\end{lemma}

\proof
$(i)\Leftrightarrow (ii)$: See \cite[Proposition 6]{kit-sch}.
$(ii)\Leftrightarrow (iii)$: See \cite[Proposition 39]{matrixnorm}.
\endproof



We conclude this section with the following helpful computational tool. 
Let $(G,p)$ be a well-positioned framework in a $d$-dimensional normed linear space $X$.
The {\em rigidity matrix} $R(G,p)$ is a matrix of linear functionals with rows indexed by the edges of $G$ and columns indexed by the vertices of $G$.
The $(e,v)$-entry for an edge $e$ and vertex $v$ is:
\[r_{e,v} = \left\{\begin{array}{ll}
\varphi_{p_v-p_w}     & \mbox{if }e=vw \\
0     & \mbox{otherwise}, 
\end{array}\right.\]
where $\varphi_{p_v-p_w}$ is the unique support functional for $p_v-p_w$. Thus, the row entries for an edge $vw$ are:
\[
\kbordermatrix{
&&&&v&&&&w&&& \\
vw&0&\cdots&0& \varphi_{p_v-p_w}&0&\cdots &0&\varphi_{p_w-p_v}&0&\cdots&0}.\]
Regarding the rigidity matrix as a linear transformation $R(G,p):X^V\to\mathbb{R}^E$ it can be shown that $R(G,p) = D(G,p)\circ df_G(p)$ where $D(G,p)$ is the diagonal matrix with rows and columns indexed by $E$ and $(vw,vw)$-entry $\|p_v-p_w\|_X$ for each edge $vw\in E$. (See \cite{kit-sch} for more details).
In practice, each support functional in the rigidity matrix $R(G,p)$ is replaced by a representing $1\times d$ row matrix, resulting in a $|E|\times d|V|$-matrix with real entries. This process is illustrated in Example \ref{ex:mono2} below.

\begin{example}
\label{ex:mono2}
Let $(G,p)$ be a framework in $\mathbb{R}^2\oplus_\infty\mathbb{R}$ and let $vw$ be an edge of $G$.
Write $p_v-p_w=(x_1,x_2,y)$.
If $vw$ is a blue edge then, by Example \ref{ex:cyl1}, the support functional for $p_v-p_w$ is represented by the row matrix $[x_1 \,\,\,x_2 \,\,\,0]$ and so the $vw$-row of the rigidity matrix has entries:
\[
\kbordermatrix{
&&&&v&&&&w&&& \\
vw&0&\cdots&0& \overbrace{x_1 \quad x_2 \quad 0}&0&\cdots 
&0&\overbrace{-x_1 \quad -x_2 \quad 0}&0&\cdots&0}.
\]
If $vw$ is a green edge then, by Example \ref{ex:cyl1}, the support functional for $p_v-p_w$ is represented by the row matrix $[0 \,\,\,0 \,\,\,y]$ and so the $vw$-row of the rigidity matrix $R(G,p)$ has entries:
\[
\kbordermatrix{
&&&&v&&&&w&&& \\
vw&0&\cdots&0&\overbrace{0 \quad 0 \quad y }&0&\cdots &0&\overbrace{0\quad 0\quad -y} &0&\cdots&0}.
\]

Consider again the  framework $(K_3,p)$ in $\mathbb{R}^2\oplus_\infty\mathbb{R}$ described in Example \ref{ex:mono}. 
By Lemma \ref{lem:wp},
the framework $(K_3,p)$ is well-positioned. The rigidity matrix is,
\[
R(K_3,p) =
\kbordermatrix{
&v_{1,1}&v_{1,2}&v_{1,3}&v_{2,1}&v_{2,2}&v_{2,3}&v_{3,1}&v_{3,2}&v_{3,3} \\
v_1v_2&0&-2&0& 0&2&0 &0&0&0\\
v_2v_3&0&0&0&  -\frac{1}{2}&2&0 &\frac{1}{2}&-2&0\\
v_1v_3&0&0&-1&  0&0&0 &0&0&1\\
}.
\]
Since $\rank df_{K_3}(p) = \rank R(K_3,p) = 3 $, the framework $(K_3,p)$ is  independent in $\mathbb{R}^2\oplus_\infty\mathbb{R}$.
\end{example}

\section{Rigidity in cylindrical normed spaces}
\label{Sec:Rigidity}
In this section, we obtain complete combinatorial characterisations of rigidity for classes of $3$-dimensional cylindrical normed spaces. We also present three rigidity preserving graph operations for $3$-dimensional cylindrical normed spaces; the $0$-extension, the $1$-extension and the vertex-split.

\subsection{Graph rigidity in cylindrical spaces}\label{S:IR}

A \emph{rigid motion} of a normed linear space $X$ is a collection $\alpha=\{\alpha_x:[-1,1]\to X\}_{x\in X}$ of continuous paths, with the following properties:
\begin{enumerate}[(a)]
\item
$\alpha_x(0)=x$ for all $x\in X$;
\item
$\alpha_x(t)$ is differentiable at $t=0$ for all $x\in X$; and
\item
$\|\alpha_x(t)-\alpha_y(t)\|_X = \|x-y\|_X$ for all $x,y\in X$ and for all $t\in [-1,1]$.
\end{enumerate} 
We write~$\mathcal{R}(X)$ for the set of all rigid motions of~$X$. 

A map $\eta : X \to X$ of the form $\eta(x) = \alpha'_x(0)$ where $\alpha \in R(X)$ is referred to as an {\em infinitesimal
rigid motion} of $X$. It can be shown that infinitesimal rigid motions are affine maps and are closed under pointwise addition and scalar multiplication. (See \cite[\S2.3]{matrixnorm}). The linear space of infinitesimal
rigid motions of $X$ is denoted $\mathcal{T}(X)$.

\begin{example}
\label{ex:rigmotions}
 Recall that $\mathcal{T}(\mathbb{R}^d)$ is the space of affine maps $\eta:\mathbb{R}^d\to \mathbb{R}^d$ of the form $\eta(x) = Ax+c$ where $A$ is a skew-symmetric $d\times d$ matrix and $c\in \mathbb{R}^d$.  
 In particular, $\dim \mathcal{T}(\mathbb{R}^d) = \frac{d(d+1)}{2}$.
 Note that for every $d$-dimensional real normed linear space $X$, $\mathcal{T}(X)$ contains the space of affine maps $\eta:X\to X$ of the form $\eta(x) = x+c$ where $c\in X$.
 Moreover, by \cite[Corollary 3.3.4]{minkowski}, $X$ is isometrically isomorphic to a normed space $Z=(\mathbb{R}^d,\|\cdot\|_Z)$ with the property that every isometry of $Z$ is an isometry of the Euclidean space $\mathbb{R}^d$. Thus, $\mathcal{T}(X)$ is linearly isomorphic to $\mathcal{T}(Z)$ where  
 $\mathcal{T}(Z)\subseteq \mathcal{T}(\mathbb{R}^d)$ and 
 $\dim \mathcal{T}(X)\in\{d,d+1,\ldots, \frac{d(d+1)}{2}\}$.
\end{example}

\begin{example}
Consider the cylindrical normed space $\mathbb{R}^2\oplus_\infty \mathbb{R}$.
Define for each $z=(x,y)\in \mathbb{R}^2\oplus_\infty \mathbb{R}$ the continuous map $\alpha_z:[-1,1]\to \mathbb{R}^2\oplus_\infty \mathbb{R}$, $\alpha_z(\theta) = (T_{\theta}(x),y)$ where the transformation $T_\theta:\mathbb{R}^2\to\mathbb{R}^2$ is clockwise rotation about the origin by the angle $2\pi\theta$.
Then the collection $\alpha=\{\alpha_z\}_{z\in \mathbb{R}^2\oplus_\infty \mathbb{R}}$ is a rigid motion of $\mathbb{R}^2\oplus_\infty \mathbb{R}$.
The induced infinitesimal rigid motion is the linear map $\eta(x,y) = (x^\perp,0)$ where $x^\perp  =(-x_2,x_1)$ for all $x=(x_1,x_2)\in\mathbb{R}^2$.
Note that in general, $\dim \mathcal{T}(X\oplus_\infty\mathbb{R}) = \dim \mathcal{T}(X) +1$ (see \cite[Theorem 44]{matrixnorm}).
\end{example}

Let $(G,p)$ be a well-positioned framework in a normed space $X$ and let $df_G(p):X^V\to \mathbb{R}^E$ be the differential of the rigidity map at $p$. 
The elements of $\ker df_G(p)$ are referred to as {\em infinitesimal flexes} of  the framework $(G,p)$.
Define
\[\mathcal{T}(G,p)=\left\{ \zeta\in X^V\mid  \zeta_v=\eta(p_v)\text{ for some $\eta\in\mathcal{T}(X)$} \right\}.\]
It can be shown that $\mathcal{T}(G,p)\subseteq \ker df_G(p)$ (see \cite[Lemma 2.3]{kit-pow}).
The elements of $\mathcal{T}(G,p)$ are referred to as the {\em trivial infinitesimal flexes} of $(G,p)$.
A well-positioned framework $(G,p)$ is \emph{infinitesimally rigid} if every infinitesimal flex of $(G,p)$ is trivial (i.e.~$\ker df_G(p)=\mathcal{T}(G,p)$); otherwise, $(G,p)$ is \emph{infinitesimally flexible}.
A well-positioned framework is \emph{minimally infinitesimally rigid} if it is both infinitesimally rigid and independent.



\begin{theorem}\label{thm:product}
    Let $(G,p)$ be a well-positioned framework in a cylindrical normed space $X \oplus_\infty \mathbb{R}$.
    The following statements are equivalent.
    \begin{enumerate}[(i)]
        \item $(G,p)$ is infinitesimally rigid (resp., independent) in $X \oplus_\infty \mathbb{R}$.
        \item The projected monochrome subframework $(G_X, p_X)$ is infinitesimally rigid (resp., independent) in $X$ and the projected monochrome subframework $(G_\mathbb{R}, p_{\mathbb{R}})$ is infinitesimally rigid (resp., independent) in $\mathbb{R}$.
    \end{enumerate}
\end{theorem}

\proof
Apply \cite[Theorem 47]{matrixnorm}.
\endproof

A graph $G$ is {\em rigid} (respectively, {\em minimally rigid}, {\em independent}) in a normed linear space $X$ if there exists a well-positioned placement $p\in X^V$ such that the framework $(G,p)$ is infinitesimally rigid (respectively,  minimally infinitesimally rigid, independent). 

\begin{example}
It is well-known that a graph is rigid (respectively, minimally rigid, independent) in $\mathbb{R}$  if and only if it is connected (respectively, a tree, a forest). 
\end{example}

\begin{example}
\label{ex:mono3}
By Theorem \ref{thm:product}, it is clear that a graph $G=(V,E)$ with fewer than $2(|V|-1)$ edges cannot be rigid in a cylindrical normed space $X\oplus_\infty\mathbb{R}$. Indeed, given any well-positioned framework $(G,p)$ in $X\oplus_\infty\mathbb{R}$, one of the two projected monochrome subframeworks will have fewer than $|V|-1$ edges and is hence disconnected. 
\end{example}

A normed space $X$ is said to be \emph{generic} if for every finite simple graph $G$ the set of regular placements of $G$ in $X$ is dense in $X^V$.
Examples of generic spaces include $\ell_q^d$ for $q\in(1,\infty)$ and $d\geq 1$ (\cite[Lemma 2.7]{kit-pow-inf}), and non-examples include polyhedral normed spaces (\cite[Lemma 16]{polyhedra}).
Importantly, Euclidean space $\mathbb{R}^d$ is a generic normed space.

\begin{proposition}
\label{mainprop}
    Let $Z=X \oplus_\infty \mathbb{R}$ be a cylindrical normed space  and let
$G= (V,E)$ be a graph with a spanning subtree (respectively, subforest) $T$.
    If $X$ is a generic space of dimension at least $2$ then the following statements are equivalent.
    \begin{enumerate}[(i)]
        \item\label{mainthmitem1} There exists an infinitesimally rigid (respectively, independent) framework $(G,p)$ in $Z$ such that the  monochrome subframeworks $(G_X,p)$ and $(G_\mathbb{R},p)$ satisfy $G_X=G-T$ and $G_\mathbb{R}=T$.
        \item\label{mainthmitem2} The subgraph $G-T$ is rigid (respectively, independent) in $X$.
    \end{enumerate}
\end{proposition}

\begin{proof}
    Suppose (\ref{mainthmitem1}) holds. By Theorem \ref{thm:product}, the projected monochrome subframework $(G-T,p_X)$ is infinitesimally rigid (respectively, independent) in $X$. Thus (\ref{mainthmitem1}) implies (\ref{mainthmitem2}).
    
    Suppose (\ref{mainthmitem2}) holds.
    By the Key Lemma (Lemma \ref{mainlem}),
    there exists a framework $(G,q)$ in $Z$ with  monochrome subframeworks $(G-T,q)$ and $(T,q)$.
     Since $K^c$ and $K^\circ$ are open sets, we must have $\kappa_p = \kappa_q$ for all placements $p$ in some open neighbourhood $U$ of $q$.
    Since $X$ is generic, the set $R$ of regular placements of $G-T$ in $X$ is both open and dense in $X^V$. The projection $\pi:Z^V\to X^V$, $\pi(p)=p_X$, is continuous and an open map and so the preimage $\pi^{-1}(R)$ is open and dense in $Z^V$. 
    Thus $U\cap \pi^{-1}(R)$ is a non-empty open subset of $Z^V$.
    Since the set of well-positioned placements is always dense (see \cite[Lemma 4.1]{dew1}), there exists a well-positioned framework $(G,p)$ in $Z$ such that $p$ lies in $U\cap \pi^{-1}(R)$.
    Note that $(G,p)$ has monochrome subframeworks $(G-T,p)$ and $(T,p)$ and  that the projected monochrome subframework $(G-T,p_X)$ is regular in $X$. 
    Since $G-T$ is rigid (respectively, independent) in $X$,
    $(G-T,p_X)$ is infinitesimally rigid (respectively, independent) in $X$.
    Similarly,
    since $T$ is a spanning subtree (respectively, subforest),
    $(T,p_\mathbb{R})$ is infinitesimally rigid (respectively, independent) in $\mathbb{R}$.
    The result now follows from Theorem \ref{thm:product}.
\end{proof}

We now present a characterisation of rigid graphs for a large class of cylindrical normed spaces of dimension $d\geq 3$.

\begin{theorem}\label{graphthm}
    Let $X \oplus_\infty \mathbb{R}$ be a cylindrical normed space and let $G= (V,E)$ be a graph. If $X$ is a generic normed space of dimension at least $2$ then the following statements are equivalent.
    \begin{enumerate}[(i)]
        \item\label{graphthmitem1} $G$ is rigid (respectively, minimally rigid, independent) in $X \oplus_\infty \mathbb{R}$.
        \item\label{graphthmitem2} $G$ is an edge-disjoint union of spanning subgraphs $H$ and $T$,
        where $H$ is rigid (respectively, minimally rigid, independent) in $X$ and $T$ is connected (respectively, a tree, a forest).
    \end{enumerate}
\end{theorem}

\begin{proof}
If $G$ is minimally rigid in  $X \oplus_\infty \mathbb{R}$
then, by Theorem \ref{thm:product}, we can set $H=G_{X}$ and $T=G_{\mathbb{R}}$.
Conversely, suppose $G$ is an edge-disjoint union of spanning subgraphs $H$ and $T$ where $H$ is minimally rigid in $X$ and $T$ is a tree.
By Proposition \ref{mainprop}, there exists a   minimally infinitesimally rigid framework $(G,p)$ in $X \oplus_\infty \mathbb{R}$ with projected monochrome subframeworks $(H,p_X)$ and $(T,p_\mathbb{R})$.
Thus, $G$ is minimally rigid in $X \oplus_\infty \mathbb{R}$.

The analogous implication $(i)\Rightarrow (ii)$ for both rigidity and independence is proved in a similar way. The converse statement $(ii)\Rightarrow (i)$ for independence is also proved in a similar way. It remains to prove the implication $(ii)\Rightarrow (i)$ in the case of rigidity.  Suppose $G$ is an edge-disjoint union of spanning subgraphs $H$ and $T$, where $H$ is rigid in $X$ and $T$ is connected. Choose a spanning subtree $T'$ of $T$ and let $H'=H \cup (T-T')$. Then $H'$ is rigid in $X$. 
By Proposition \ref{mainprop}, there exists an infinitesimally rigid framework $(G,p)$ in $X \oplus_\infty \mathbb{R}$ with projected monochrome subframeworks $(H',p_X)$ and $(T',p_\mathbb{R})$.
Thus, $G$ is rigid in $X \oplus_\infty \mathbb{R}$.
\end{proof}



Before stating our main theorem on rigid graphs in $3$-dimensional cylindrical normed spaces we recall the following terminology and results.
Let $k$ and $\ell$ be integers with $k\geq 2$ and $\ell\in\{k, k+1\}$.
A graph $G=(V,E)$ is said to be \emph{$(k,\ell)$-sparse} if for every subgraph $H=(V(H),E(H))$ of $G$ with at least one edge we have $|E(H)| \leq k|V(H)| - \ell$.
A graph $G$ is said to be \emph{$(k,\ell)$-tight} if it is $(k,\ell)$-sparse and satisfies $|E|=k|V|-\ell$.

\begin{theorem}\label{t:2d}
Let $X$ be a 2-dimensional normed space and let $G=(V,E)$ be a  graph.
\begin{enumerate}[(a)]
\item If $X$ is isometrically isomorphic to $\mathbb{R}^2$ then the following statements are equivalent.
\begin{enumerate}[(i)]
\item $G$ is minimally rigid (respectively, independent) in $X$.
\item $G$ is $(2,3)$-tight (respectively, $(2,3)$-sparse).
\end{enumerate}
\item If $X$ is not isometrically isomorphic to $\mathbb{R}^2$ then the following statements are equivalent.
\begin{enumerate}[(i)]
\item $G$ is minimally rigid (respectively, independent) in $X$.
\item $G$ is $(2,2)$-tight (respectively, $(2,2)$-sparse). 
\end{enumerate}
\end{enumerate}
\end{theorem}

\proof
Part (a) is due to Pollaczek-Geiringer \cite{pollaczek27} and was later proved independently by Laman \cite{laman}.
Part (b) is proved in  \cite{dew2} (see also \cite{polyhedra,kit-pow}).
\endproof

 \begin{theorem}[Nash-Williams \cite{nashwill}]\label{thm:nashwill}
     A graph is $(k,k)$-tight if and only if it can be expressed as a union of $k$ edge-disjoint spanning trees.
 \end{theorem}

We  now present our main combinatorial characterisations of rigid graphs for $3$-dimensional cylindrical normed spaces.
Part (a) of the following theorem gives an affirmative answer to Conjecture 59(a) in \cite{matrixnorm}. 

\begin{theorem}\label{mainthm}
    Let $X$ be a generic normed plane and let $G=(V,E)$ be a  graph.
    \begin{enumerate}[(a)]
        \item If $X$ is isometrically isomorphic to $\mathbb{R}^2$ then the following statements are equivalent.
        \begin{enumerate}[(i)]
            \item\label{maincor1item1} $G$ is minimally rigid in the cylindrical normed space $X \oplus_\infty \mathbb{R}$.
            \item\label{maincor1item2} $G$ is an edge disjoint union of spanning subgraphs $H$ and $T$,
            where $H$ is $(2,3)$-tight and $T$ is a tree.
        \end{enumerate}
        \item If $X$ is not isometrically isomorphic to $\mathbb{R}^2$ then the following statements are equivalent.
        \begin{enumerate}[(i)]
            \item\label{maincor1item3} $G$ is minimally rigid in the cylindrical normed space $X \oplus_\infty \mathbb{R}$.
            \item $G$ is an edge disjoint union of three spanning trees.
            \item\label{maincor1item4} $G$ is $(3,3)$-tight.
        \end{enumerate}
    \end{enumerate}
\end{theorem}

\begin{proof}
Part $(a)$ follows from Theorem \ref{graphthm} and Theorem \ref{t:2d}(a).
For part $(b)$, note that by Theorem \ref{t:2d}(b)
    the minimally rigid graphs in $X$ are precisely the $(2,2)$-tight graphs. By Theorem \ref{thm:nashwill}, 
    a graph is $(2,2)$-tight if and only if it is an edge-disjoint union of two spanning trees. Thus the equivalence of $(i)$ and $(ii)$ follows from Theorem \ref{graphthm}. The equivalence of $(ii)$ and $(iii)$ is a special case of Theorem \ref{thm:nashwill}
    \end{proof}

\begin{remark}\label{r:3,4}
    It is evident from Theorem \ref{mainthm}(a) that minimally rigid graphs in $\mathbb{R}^2\oplus_\infty\mathbb{R}$ must be $(3,4)$-tight. In \cite{matrixnorm} it was observed that the converse statement is not true; see Figure \ref{fig:34eg} for example. By a result of Frank and Szeg\"{o} (see \cite[Theorem 1.9]{frankszego03}), 
    every $(3,4)$-tight graph can be constructed from a pair of parallel edges by a sequence of 3-dimensional multigraph 0-, 1- and 2-extensions in which no triples of parallel edges arise.  
    Similarly, if $X$ is a generic normed plane which is not isometrically isomorphic to $\mathbb{R}^2$ then, by Theorem \ref{mainthm}(b), the minimally rigid graphs in $X \oplus_\infty \mathbb{R}$ are $(3,3)$-tight. By \cite[Theorem 1.8]{frankszego03}, every $(3,3)$-tight graph can be constructed from a triple of parallel edges by a sequence of 3-dimensional multigraph  0-, 1- and 2-extensions.
    Note that these constructions  do not guarantee that every graph in the sequence is itself a minimally rigid graph (or indeed a simple graph). 
\end{remark}

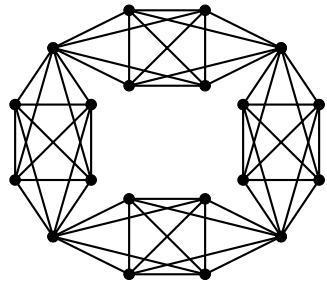
\begin{figure}[ht]
 \centering
 \begin{tabular}{  l  r  }

 \begin{tikzpicture}[scale=0.5]
  \clip (-6.25,-6.5) rectangle (8.25cm,1.5cm); 
  \coordinate (A1) at (-1,-1);
  \coordinate (A2) at (1,-1);
  \coordinate (A3) at (1,1);
  \coordinate (A4) at (-1,1);
  \coordinate (A5) at (-3,0);
  \coordinate (A6) at (3,0);
	
  \draw[thick] (A1) -- (A5) -- (A4) -- (A2) -- (A3) -- (A1) -- (A4) -- (A3) -- (A6) --
	(A2) -- (A1);
	\draw[thick] (A2) -- (A5) -- (A3);
	\draw[thick] (A4) -- (A6) -- (A1);
  
  \node[draw,circle,inner sep=1.4pt,fill] at (A1) {};
  \node[draw,circle,inner sep=1.4pt,fill] at (A2) {};
  \node[draw,circle,inner sep=1.4pt,fill] at (A3) {};
  \node[draw,circle,inner sep=1.4pt,fill] at (A4) {};
  \node[draw,circle,inner sep=1.4pt,fill] at (A5) {};
  \node[draw,circle,inner sep=1.4pt,fill] at (A6) {};
	
	\coordinate (B1) at (-4,-3.5);
  \coordinate (B2) at (-2,-3.5);
  \coordinate (B3) at (-2,-1.5);
  \coordinate (B4) at (-4,-1.5);
  \coordinate (B5) at (-3,0);
  \coordinate (B6) at (-3,-5);
	
  \draw[thick] (B1) -- (B5) -- (B4) -- (B2) -- (B3) -- (B1) -- (B4) -- (B3) -- (B6) --
	(B2) -- (B1);
	\draw[thick] (B2) -- (B5) -- (B3);
	\draw[thick] (B4) -- (B6) -- (B1);
  
  \node[draw,circle,inner sep=1.4pt,fill] at (B1) {};
  \node[draw,circle,inner sep=1.4pt,fill] at (B2) {};
  \node[draw,circle,inner sep=1.4pt,fill] at (B3) {};
  \node[draw,circle,inner sep=1.4pt,fill] at (B4) {};
  \node[draw,circle,inner sep=1.4pt,fill] at (B5) {};
  \node[draw,circle,inner sep=1.4pt,fill] at (B6) {};
	
	\coordinate (C1) at (-1,-6);
  \coordinate (C2) at (1,-6);
  \coordinate (C3) at (1,-4);
  \coordinate (C4) at (-1,-4);
  \coordinate (C5) at (-3,-5);
  \coordinate (C6) at (3,-5);
	
  \draw[thick] (C1) -- (C5) -- (C4) -- (C2) -- (C3) -- (C1) -- (C4) -- (C3) -- (C6) --
	(C2) -- (C1);
	\draw[thick] (C2) -- (C5) -- (C3);
	\draw[thick] (C4) -- (C6) -- (C1);
  
  \node[draw,circle,inner sep=1.4pt,fill] at (C1) {};
  \node[draw,circle,inner sep=1.4pt,fill] at (C2) {};
  \node[draw,circle,inner sep=1.4pt,fill] at (C3) {};
  \node[draw,circle,inner sep=1.4pt,fill] at (C4) {};
  \node[draw,circle,inner sep=1.4pt,fill] at (C5) {};
  \node[draw,circle,inner sep=1.4pt,fill] at (C6) {};
	
	\coordinate (D1) at (4,-3.5);
  \coordinate (D2) at (2,-3.5);
  \coordinate (D3) at (2,-1.5);
  \coordinate (D4) at (4,-1.5);
  \coordinate (D5) at (3,0);
  \coordinate (D6) at (3,-5);
	
  \draw[thick] (D1) -- (D5) -- (D4) -- (D2) -- (D3) -- (D1) -- (D4) -- (D3) -- (D6) --
	(D2) -- (D1);
	\draw[thick] (D2) -- (D5) -- (D3);
	\draw[thick] (D4) -- (D6) -- (D1);
  
  \node[draw,circle,inner sep=1.4pt,fill] at (D1) {};
  \node[draw,circle,inner sep=1.4pt,fill] at (D2) {};
  \node[draw,circle,inner sep=1.4pt,fill] at (D3) {};
  \node[draw,circle,inner sep=1.4pt,fill] at (D4) {};
  \node[draw,circle,inner sep=1.4pt,fill] at (D5) {};
  \node[draw,circle,inner sep=1.4pt,fill] at (D6) {};
  \end{tikzpicture}

\end{tabular}
\caption{A $(3,4)$-tight graph $G$ which cannot be decomposed into a spanning $(2,3)$-tight subgraph and a spanning tree. In particular, $G$ is not minimally rigid in $\mathbb{R}^2\oplus_\infty\mathbb{R}$.
The graph $G$ has an edge-connectivity of 5.}
\label{fig:34eg}
\end{figure}

\subsection{3-dimensional graph operations}

For a graph $G=(V,E)$,
we define the following operations for any positive integer $d$.
\begin{enumerate}[(i)]
    \item A \emph{$d$-dimensional 0-extension} of $G$ is a graph $G'$ obtained from $G$ by adjoining a new vertex which is adjacent to $d$ distinct vertices of $G$.
    \item  A \emph{$d$-dimensional 1-extension} of $G$ is a graph $G'$ obtained from $G$ by deleting an edge $xy$  and adjoining a new vertex  which is adjacent to $x$, $y$ and $d-1$ other vertices in $G$.
    \item Let $w$ be a vertex of $G$ with neighbourhood $N_G(w)$ containing at least $d-1$ vertices labelled $v_1,\ldots, v_{d-1}$.
    A \emph{$d$-dimensional vertex-split} of $G$ is a graph $G'$ obtained from $G$ by adjoining a new vertex $w'$ which is adjacent to $w$ and $v_1,\ldots,v_{d-1}$
    and, for every $v \in N_G(w) \setminus \{v_1,\ldots,v_{d-1}\}$,  either leaving the edge $wv$ unchanged or replacing it with $w'v$.
\end{enumerate}
In what follows we will only require the cases $d=2$ and $d=3$.
See Figure \ref{fig:extmoves} for illustrations of these operations.

\begin{figure}[ht]
 \centering
\begin{tikzpicture}[scale=0.7]
\begin{scope}[shift={(0,3)},scale=1]
  \coordinate (A1) at (-0.5,0);
  \coordinate (A2) at (0.5,0);

  \draw[thick] (0,0) ellipse (1.25 and 0.75);

  \draw[-latex,thick] (1.5,0) -- (2.5,0);
  
  \coordinate (B1) at (3.5,0);
  \coordinate (B2) at (4.5,0);
  \coordinate (B3) at (4,1.5);

  \draw[thick] (4,0) ellipse (1.25 and 0.75);
  
  \draw[thick] (B1) -- (B3) -- (B2);
  
  \node[draw,circle,inner sep=1pt,fill] at (A1) {};
  \node[draw,circle,inner sep=1pt,fill] at (A2) {};
  
  \node[draw,circle,inner sep=1pt,fill] at (B1) {};
  \node[draw,circle,inner sep=1pt,fill] at (B2) {};
  \node[draw,circle,inner sep=1pt,fill] at (B3) {};
\end{scope}

\begin{scope}[shift={(8,3)},scale=1]
  \coordinate (A1) at (-0.5,0);
  \coordinate (A2) at (0.5,0);
  \coordinate (A3) at (0,-0.25);

  \draw[thick] (0,0) ellipse (1.25 and 0.75);

  \draw[-latex,thick] (1.5,0) -- (2.5,0);
  
  \coordinate (B1) at (3.5,0);
  \coordinate (B2) at (4.5,0);
  \coordinate (B3) at (4,-0.25);
  \coordinate (B4) at (4,1.5);

  \draw[thick] (4,0) ellipse (1.25 and 0.75);
  
  \draw[thick] (A1) -- (A2);
  \draw[thick] (B1) -- (B4) -- (B2);
  \draw[thick] (B3) -- (B4);
  
  \node[draw,circle,inner sep=1pt,fill] at (A1) {};
  \node[draw,circle,inner sep=1pt,fill] at (A2) {};
  \node[draw,circle,inner sep=1pt,fill] at (A3) {};
  
  \node[draw,circle,inner sep=1pt,fill] at (B1) {};
  \node[draw,circle,inner sep=1pt,fill] at (B2) {};
  \node[draw,circle,inner sep=1pt,fill] at (B3) {};
  \node[draw,circle,inner sep=1pt,fill] at (B4) {};
\end{scope}

\begin{scope}[shift={(16,3)},scale=1]
  \coordinate (A1) at (-0.5,0);
  \coordinate (A2) at (0.5,0);

  \draw[thick] (0,0) ellipse (1.25 and 0.75);
  
  \draw[thick] (A1) -- (A2);

  \draw[thick] (A1) -- (-0.3,0.2);
  \draw[thick] (A1) -- (-0.3,-0.2);
  \draw[thick] (A1) -- (-0.7,0.2);
  \draw[thick] (A1) -- (-0.7,-0.2);

  \draw[-latex,thick] (1.5,0) -- (2.5,0);

  \coordinate (B1) at (3.5,0.2);
  \coordinate (B2) at (4.5,0);
  \coordinate (B3) at (3.5,-0.2);

  \draw[thick] (4,0) ellipse (1.25 and 0.75);

  \draw[thick] (B1) -- (B2);
  \draw[thick] (B2) -- (B3);
  \draw[thick] (B3) -- (B1);

  \draw[thick] (B1) -- (3.3,0.4);
  \draw[thick] (B1) -- (3.7,0.4);
  \draw[thick] (B3) -- (3.3,-0.4);
  \draw[thick] (B3) -- (3.7,-0.4);
  
  \node[draw,circle,inner sep=1pt,fill] at (A1) {};
  \node[draw,circle,inner sep=1pt,fill] at (A2) {};
  
  \node[draw,circle,inner sep=1pt,fill] at (B1) {};
  \node[draw,circle,inner sep=1pt,fill] at (B2) {};
  \node[draw,circle,inner sep=1pt,fill] at (B3) {};
\end{scope}

\begin{scope}[shift={(0,0)},scale=1]
  \coordinate (A1) at (-0.6,0);
  \coordinate (A2) at (0.6,0);
  \coordinate (A3) at (0,0);

  \draw[thick] (0,0) ellipse (1.25 and 0.75);

  \draw[-latex,thick] (1.5,0) -- (2.5,0);
  
  \coordinate (B1) at (3.4,0);
  \coordinate (B2) at (4.6,0);
  \coordinate (B3) at (4,0);
  \coordinate (B4) at (4,1.5);

  \draw[thick] (4,0) ellipse (1.25 and 0.75);
  
  \draw[thick] (B1) -- (B4) -- (B2);
  \draw[thick] (B3) -- (B4);
  
  \node[draw,circle,inner sep=1pt,fill] at (A1) {};
  \node[draw,circle,inner sep=1pt,fill] at (A2) {};
  \node[draw,circle,inner sep=1pt,fill] at (A3) {};
  
  \node[draw,circle,inner sep=1pt,fill] at (B1) {};
  \node[draw,circle,inner sep=1pt,fill] at (B2) {};
  \node[draw,circle,inner sep=1pt,fill] at (B3) {};
  \node[draw,circle,inner sep=1pt,fill] at (B4) {};
\end{scope}

\begin{scope}[shift={(8,0)},scale=1]
  \coordinate (A1) at (-0.6,0);
  \coordinate (A2) at (0.6,0);
  \coordinate (A3) at (0.2,-0.25);
  \coordinate (A4) at (-0.2,-0.25);

  \draw[thick] (0,0) ellipse (1.25 and 0.75);

  \draw[-latex,thick] (1.5,0) -- (2.5,0);
  
  \coordinate (B1) at (3.4,0);
  \coordinate (B2) at (4.6,0);
  \coordinate (B3) at (4.2,-0.25);
  \coordinate (B4) at (3.8,-0.25);
  \coordinate (B5) at (4,1.5);

  \draw[thick] (4,0) ellipse (1.25 and 0.75);
  
  \draw[thick] (A1) -- (A2);
  \draw[thick] (B1) -- (B5) -- (B2);
  \draw[thick] (B3) -- (B5);
  \draw[thick] (B4) -- (B5);
  
  \node[draw,circle,inner sep=1pt,fill] at (A1) {};
  \node[draw,circle,inner sep=1pt,fill] at (A2) {};
  \node[draw,circle,inner sep=1pt,fill] at (A3) {};
  \node[draw,circle,inner sep=1pt,fill] at (A4) {};
  
  \node[draw,circle,inner sep=1pt,fill] at (B1) {};
  \node[draw,circle,inner sep=1pt,fill] at (B2) {};
  \node[draw,circle,inner sep=1pt,fill] at (B3) {};
  \node[draw,circle,inner sep=1pt,fill] at (B4) {};
  \node[draw,circle,inner sep=1pt,fill] at (B5) {};
\end{scope}

\begin{scope}[shift={(16,0)},scale=1]
  \coordinate (A1) at (-0.6,0);
  \coordinate (A2) at (0,0);
  \coordinate (A3) at (0.6,0);

  \draw[thick] (0,0) ellipse (1.25 and 0.75);
  
  \draw[thick] (A1) -- (A2) -- (A3);

  \draw[thick] (A2) -- (0.2,0.2);
  \draw[thick] (A2) -- (0.2,-0.2);
  \draw[thick] (A2) -- (-0.2,0.2);
  \draw[thick] (A2) -- (-0.2,-0.2);

  \draw[-latex,thick] (1.5,0) -- (2.5,0);

  \coordinate (B1) at (3.4,0);
  \coordinate (B2) at (4,0.2);
  \coordinate (B3) at (4.6,0);
  \coordinate (B4) at (4,-0.2);

  \draw[thick] (4,0) ellipse (1.25 and 0.75);

  \draw[thick] (B1) -- (B2);
  \draw[thick] (B2) -- (B3);
  \draw[thick] (B1) -- (B4);
  \draw[thick] (B3) -- (B4);
  \draw[thick] (B2) -- (B4);

  \draw[thick] (B2) -- (3.8,0.4);
  \draw[thick] (B2) -- (4.2,0.4);
  \draw[thick] (B4) -- (3.8,-0.4);
  \draw[thick] (B4) -- (4.2,-0.4);
  
  \node[draw,circle,inner sep=1pt,fill] at (A1) {};
  \node[draw,circle,inner sep=1pt,fill] at (A2) {};
  \node[draw,circle,inner sep=1pt,fill] at (A3) {};
  
  \node[draw,circle,inner sep=1pt,fill] at (B1) {};
  \node[draw,circle,inner sep=1pt,fill] at (B2) {};
  \node[draw,circle,inner sep=1pt,fill] at (B3) {};
  \node[draw,circle,inner sep=1pt,fill] at (B4) {};
\end{scope}
\end{tikzpicture}
\caption{
(Top row) From left to right: a 2-dimensional 0-extension, a 2-dimensional 1-extension, a 2-dimensional vertex splitting operation.
(Bottom row) From left to right: a 3-dimensional 0-extension, a 3-dimensional 1-extension, a 3-dimensional vertex split.
}
\label{fig:extmoves}
\end{figure}
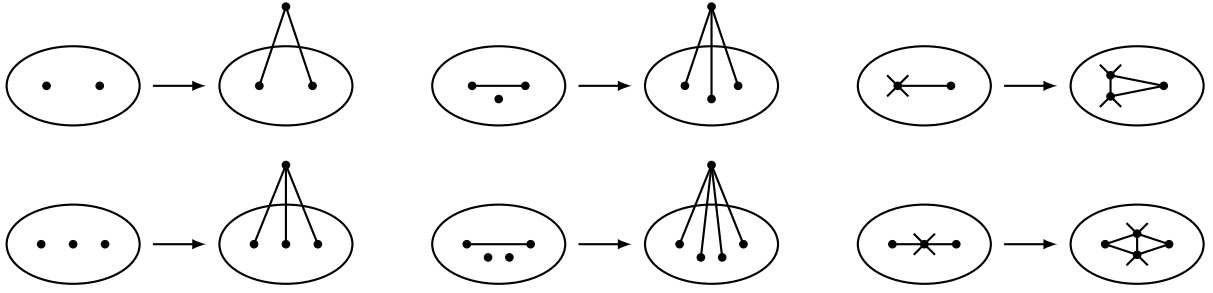

In \cite[\S4]{DewKitNix} it is shown that each of the $d$-dimensional operations described above preserves both independence and rigidity in any $d$-dimensional strictly convex and smooth normed space. The proofs presented in \cite{DewKitNix} are geometric and not applicable to cylindrical normed spaces. However, using the results of the previous subsection, we can now present entirely combinatorial arguments which prove analogous statements for  $3$-dimensional cylindrical normed spaces.

We will require the following well-known and easily verifiable lemma.

\begin{lemma}\label{l:2dgraphops}
    Let $G=(V,E)$ be a $(2,k)$-sparse (respectively, $(2,k)$-tight) graph where $k \in \{2,3\}$ and $|V| \geq 3$.
    Suppose a graph $G'$ is obtained from $G$ by either: 
    \begin{enumerate}[(a)]
    \item a 2-dimensional 0-extension, 
    \item a 2-dimensional 1-extension, or, 
    \item a 2-dimensional vertex split. 
    \end{enumerate}
    Then $G'$ is $(2,k)$-sparse (respectively, $(2,k)$-tight) also.
\end{lemma}

\begin{proposition}\label{p:2d0ext}
    Let $X$ be a generic normed plane.
    Then 3-dimensional 0-extensions preserve independence and rigidity in the cylindrical normed space $X \oplus_\infty \mathbb{R}$.
\end{proposition}

\begin{proof}
    Let $G$ be an independent graph in $X \oplus_\infty \mathbb{R}$.
    By Theorem \ref{graphthm},
    $G$ can be expressed as an edge disjoint union of spanning subgraphs $H$ and $T$,
    where $H$ is independent in $X$ and $T$ is a forest. 
    Let $G'$ be formed from $G$ by a $3$-dimensional 0-extension that adds a vertex $v_0$ and edges $v_0 v_1,v_0v_2,v_0v_3$.
    Define $H' = H + v_0 + \{v_0 v_1, v_0v_2\}$ and $T' = T + v_0 + v_0 v_{3}$. Note that $H'$ is a 2-dimensional 0-extension of $H$.
    Thus, $H'$ is independent in $X$ by Theorem \ref{t:2d} and Lemma \ref{l:2dgraphops}.
    As $T'$ is a forest and $G' = H' \cup T'$,
    it follows from Theorem \ref{graphthm} that $G'$ is independent in $X \oplus_\infty \mathbb{R}$.
    The analogous statement for rigidity is proved in a similar way.
\end{proof}

\begin{proposition}\label{p:2d1ext}
    Let $X$ be a generic normed plane.
    Then 3-dimensional 1-extensions preserve independence and rigidity in  the cylindrical normed space $X \oplus_\infty \mathbb{R}$.
\end{proposition}

\begin{proof}
    Let $G'$ be obtained from a graph $G$ by a $3$-dimensional 1-extension that adds a vertex $v_0$, deletes an edge $xy$ and adds edges $v_0 x, v_0 y$ and the edges $v_0 v_1, v_0 v_2$.
    
    Suppose $G$ is an independent graph in $X \oplus_\infty \mathbb{R}$.
    By Theorem \ref{graphthm},
    $G$ can be expressed as an edge disjoint union of spanning subgraphs $H$ and $T$,
    where $H$ is independent in $X$ and $T$ is a forest.
    If $xy \in E(T)$ then define $H' = H + v_0 + \{v_0 v_1, v_0 v_2\}$ and $T' = T + v_0 + \{v_0 x,v_0 y\}$.
    Note that $H'$ is a 2-dimensional 0-extension of $H$. 
    If  $xy \in E(H)$ then define $H' = H + v_0 + \{v_0 x,v_0 y, v_0 v_1\}$ and $T' = T + v_0 + \{v_0 v_2\}$.
    Note that $H'$ is a 2-dimensional 1-extension of $H$.  In either case, $H'$ is independent in $X$ (by Lemma \ref{l:2dgraphops} and Theorem \ref{t:2d}) and $T'$ is a forest.
    As $G' = H' \cup T'$,
    it follows from Theorem \ref{graphthm} that $G'$ is independent in $X \oplus_\infty \mathbb{R}$.
    
    Now suppose $G$ is a rigid graph in $X \oplus_\infty \mathbb{R}$.
    If $G$ is minimally rigid in $X \oplus_\infty \mathbb{R}$ then an argument analogous to the one above shows that $G'$ is also minimally rigid in $X \oplus_\infty \mathbb{R}$.    
    If $G$ is not minimally rigid in $X \oplus_\infty \mathbb{R}$ then choose a minimally rigid spanning subgraph $\tilde{G}$ of $G$. 
    If $xy\in E(\tilde{G})$ then the 3-dimensional 1-extension of $\tilde{G}$ at $xy$ that adds the vertex $v_0$ and edges $v_0 x$, $v_0 y$, $v_0 v_1$, $v_0 v_2$ is a spanning subgraph of $G'$ and is also minimally rigid in $X \oplus_\infty \mathbb{R}$.
      If $xy\notin E(\tilde{G})$ then the 3-dimensional 0-extension of $\tilde{G}$ that adds the vertex $v_0$ and edges $v_0 x$, $v_0 y$, $v_0 v_1$ is a spanning subgraph of $G'$ and is minimally rigid in $X \oplus_\infty \mathbb{R}$ by Proposition \ref{p:2d0ext}.
     Hence, in either case,  $G'$ is rigid in $X \oplus_\infty \mathbb{R}$.
\end{proof}

To prove our final result in this section we will require the following two lemmas.

\begin{lemma}\label{l:colourswitch}
    Let $X$ be a normed plane and let $G=(V,E)$ be a graph which is an edge-disjoint union of spanning subgraphs $H$ and $T$,
    where $H$ is independent in $X$ and $T$ is a tree.
    Suppose $H+e$ is not independent in $X$ for some edge $e$ of $T$.
    Then there exists an edge $f$ of $H$ such that $H-f +e$ is independent in $X$ and $T-e +f$ is a tree.
\end{lemma}

\begin{proof}
    By Theorem \ref{t:2d},
    there exists $k \in \{2,3\}$ such that a graph is independent in $X$ if and only if it is $(2,k)$-sparse.
    Thus, $H$ is $(2,k)$-sparse but $H+e$ is not.
    Let $H'$ be a minimal element (under inclusion) of the set of all subgraphs $J \subset H +e$ that are not $(2,k)$-sparse.
    Then $|E(H')|=2|V(H')|-k+1$ and, by minimality,  $H'-g$ is $(2,k)$-tight for every edge $g \in E(H')$.
    It is immediate that $H'$ must contain $e$ (since $H$ is $(2,k)$-sparse).
    Furthermore,
    the graph $H'$ is also the unique subgraph of $H+e$ with this property;
    indeed if another graph $H''$ had this property then it too would contain $e$ and the subgraph $(H' \cup H'') -e$ of $H$ would not be $(2,k)$-sparse.    
    Label the vertices of the two connected components of $T-e$ by $V_1,V_2$.
    Since $e$ connects $V_1$ and $V_2$,
    $H'$ contains vertices from both $V_1$ and $V_2$.
    As $H'-e$ is $(2,k)$-tight,
    it is connected;
    this follows from the observation that a $(2,k)$-sparse graph with $n$ vertices and $c$ connected components has at most $2n-ck$ edges.
    Hence there exists another edge $f \in E(H')-e$ which connects $V_1$ and $V_2$.
    Following from the uniqueness of $H'$,
    we have that $H-f +e$ is $(2,k)$-sparse and $T-e+f$ is connected.
    Since $|E(T-e+f)|=|E(T)|$, the graph $T-e+f$ is a tree also.
\end{proof}

Let $G=(V,E)$ be a graph with a vertex $w\in V$ and denote the neighbourhood of $w$ by $N_G(w)$.
Suppose $N_G(w)$ contains distinct vertices $v_1,v_2$.
A \emph{spider split} (or {\em vertex to 4-cycle operation}) of $G$ is a graph $G'$ obtained from $G$ by adjoining a new vertex $w'$ adjacent to $v_1,v_2$ and,
for every $v \in N_G(w) \setminus \{v_1,v_2\}$,
either leaving the edge $wv$ unchanged or replacing it with $w'v$.
See Figure \ref{fig:spider} for an illustration of this operation.

\begin{figure}[ht]
 \centering
\begin{tikzpicture}[scale=0.7]
  \coordinate (A1) at (-0.6,0);
  \coordinate (A2) at (0,0);
  \coordinate (A3) at (0.6,0);

  \draw[thick] (0,0) ellipse (1.25 and 0.75);
  
  \draw[thick] (A1) -- (A2) -- (A3);

  \draw[thick] (A2) -- (0.2,0.2);
  \draw[thick] (A2) -- (0.2,-0.2);
  \draw[thick] (A2) -- (-0.2,0.2);
  \draw[thick] (A2) -- (-0.2,-0.2);

  \draw[-latex,thick] (1.5,0) -- (2.5,0);

  \coordinate (B1) at (3.4,0);
  \coordinate (B2) at (4,0.2);
  \coordinate (B3) at (4.6,0);
  \coordinate (B4) at (4,-0.2);

  \draw[thick] (4,0) ellipse (1.25 and 0.75);

  \draw[thick] (B1) -- (B2);
  \draw[thick] (B2) -- (B3);
  \draw[thick] (B1) -- (B4);
  \draw[thick] (B3) -- (B4);

  \draw[thick] (B2) -- (3.8,0.4);
  \draw[thick] (B2) -- (4.2,0.4);
  \draw[thick] (B4) -- (3.8,-0.4);
  \draw[thick] (B4) -- (4.2,-0.4);
  
  \node[draw,circle,inner sep=1pt,fill] at (A1) {};
  \node[draw,circle,inner sep=1pt,fill] at (A2) {};
  \node[draw,circle,inner sep=1pt,fill] at (A3) {};
  
  \node[draw,circle,inner sep=1pt,fill] at (B1) {};
  \node[draw,circle,inner sep=1pt,fill] at (B2) {};
  \node[draw,circle,inner sep=1pt,fill] at (B3) {};
  \node[draw,circle,inner sep=1pt,fill] at (B4) {};
\end{tikzpicture}
\caption{
The spider split operation.
}
\label{fig:spider}
\end{figure}
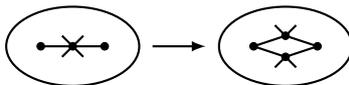

The following result is  folklore.

\begin{lemma}\label{l:extra2dgraphops}
    Let $G=(V,E)$ be a $(2,k)$-sparse (respectively, $(2,k)$-tight) graph where $k \in \{2,3\}$ and $|V| \geq 3$.
    Suppose $G'$ is obtained from $G$ by a spider split.
    Then $G'$ is $(2,k)$-sparse (respectively, $(2,k)$-tight) also.
\end{lemma}

\begin{proposition}\label{p:2dvertsplit}
    Let $X$ be a generic normed plane.
    Then 3-dimensional vertex splits preserve independence and rigidity in the cylindrical normed space $X \oplus_\infty \mathbb{R}$.
\end{proposition}

\begin{proof}
    Let $G=(V,E)$ be an independent graph in $X \oplus_\infty \mathbb{R}$.
    Without loss of generality, we may suppose that $G$ is connected.
    By Theorem \ref{graphthm},
    $G$ can be expressed as an edge disjoint union of spanning subgraphs $H$ and $T$,
    where $H$ is independent in $X$ and $T$ is a forest.
    By removing edges from $H$ and adding to $T$,
    we may suppose that $T$ is a tree (note that this is possible since $G$ is connected).
    Let $w \in V$ be a vertex adjacent to at least two distinct vertices $v_1,v_2\in V$ 
    and let $S\subseteq N_G(w) \setminus \{v_1,v_2\}$.
    Let $G'$ be formed from $G$ by a 3-dimensional vertex split at the vertex $w$ that adds a vertex $w'$, adds the edges $ww',w'v_1, w'v_2$, and for each $v \in S$ replaces the edge $wv$ with $w'v$.
    
    If both $wv_1,wv_2$ are edges of $T$,
    we refactor $H$ and $T$ as follows:
    if $H + wv_1$ is independent in $X$ then replace $H$ with the independent graph $H+wv_1$ and $T$ with the forest $T-wv_1$ (it will not matter for the rest of the proof whether $T$ is connected);
    if $H + wv_1$ is not independent in $X$ then apply Lemma \ref{l:colourswitch} to obtain an edge $f$, replace $H$ with the independent graph $H+wv_1-f$ and replace $T$ with the tree $T-wv_1+f$.
    
    Up to relabelling $v_1,v_2$,
    there are now two cases to check:
    $(i)$ $wv_1,wv_2 \in E(H)$, and
    $(ii)$ $wv_1 \in E(H)$ and $wv_2 \in E(T)$.
    We first define $H'',T''$ to be the spanning subgraphs of $G'$ where
    \begin{align*}
        E(H'') &= \Big(E(H) + \{ w'v: v \in S,~ wv \in E(H) \}\Big) - \Big(\{wv : v \in S \} + \{wv_1,wv_2\}\Big)\\
        E(T'') &= \Big(E(T) + \{ w'v: v \in S,~ wv \in E(T) \}\Big) - \Big(\{wv : v \in S \} + \{wv_1,wv_2\}\Big).
    \end{align*}
    
    Case $(i)$: 
    First suppose $wv_1,wv_2 \in E(H)$.
    Define $H' = H'' +\{wv_1,wv_2,w'v_1,w'v_2\}$ and $T' = T'' + ww'$.
    Then $H'$ is a spider split of $H$ and $T'$ contains no cycles.
    By Lemma \ref{l:extra2dgraphops} and  Theorem \ref{t:2d},
    $H'$ is independent in $X$.
    Hence by Theorem \ref{graphthm},
    $G'$ is independent in $X \oplus_\infty \mathbb{R}$.
    
    Case $(ii)$: 
    Now suppose $wv_1 \in E(H)$ and $wv_2 \in E(T)$.
    Define $H' = H'' +\{wv_1,w'v_1,ww'\}$ and $T' = T''  +\{wv_2,w'v_2\}$.
    Then $H'$ is a 2-dimensional vertex split of $H$ and $T'$ contains no cycles.
    By Lemma \ref{l:2dgraphops} and Theorem \ref{t:2d},
    $H'$ is independent in $X$.
    Hence by Theorem \ref{graphthm},
    $G'$ is independent in $X \oplus_\infty \mathbb{R}$.
    
    
    
    Now suppose that $G$ is rigid in $X \oplus_\infty \mathbb{R}$. 
    By Theorem \ref{graphthm},
    $G$ contains edge-disjoint spanning subgraphs $H$ and $T$,
    where $H$ is minimally rigid in $X$ and $T$ is a tree.
    Define $G_0 := H \cup T$.
    If the edges $wv_1$ and $wv_2$ are contained in $G_0$ then let $G_0'$ be the subgraph of $G'$ that is obtained when the 3-dimensional vertex-split operation on $G$ is applied to the subgraph $G_0$.
    By applying the above arguments to $G_0$, $H$ and $T$,
    we obtain two edge-disjoint spanning subgraphs $H'$ and $T'$ of $G_0'$ where $H'$ is minimally rigid in $X$ and $T'$ is a tree.
    Thus $G_0'$, and hence also $G'$, is rigid in $X \oplus_\infty \mathbb{R}$ by Theorem \ref{graphthm}.
    Now suppose that at least one of the edges $wv_1, wv_2$ is not an edge in $G_0$,
    i.e., the set $F := \{wv_1, wv_2\} \setminus E(G_0)$ is non-empty.
    The graph $T + F$ contains a spanning tree $S$ that contains the edge set $F$.
    We now replace $T$ with $S$ in the definition of $G_0$ and apply the previous method.
\end{proof}

\section{Examples}
\label{S:Examples}
\subsection{Braced triangulations of the sphere}

In the following, we denote the complete graph on six vertices minus any edge by $K_6-e$,
and the graph obtained by gluing two copies of $K_5$ at three vertices by $K_5 \cup_{K_3} K_5$ (see Figure \ref{fig:k5glued}).
The following result is proved in \cite{CruKasKitSch}.

\begin{proposition}
\label{p:CruKasKitSch}
    Let $G$ be a simple graph formed from a triangulation of the 2-sphere plus two extra edges.
    Then $G$ can be constructed from either $K_6-e$ or $K_5 \cup_{K_3} K_5$ by a sequence of 3-dimensional vertex splitting operations.
\end{proposition}

\begin{figure}[ht!]
 \centering
\begin{tikzpicture}[scale=0.8]
\begin{scope}[shift={(0,0)},scale=1]
  \coordinate (A1) at (-2,0);
  \coordinate (A2) at (-2,1);
  \coordinate (A3) at (0,-0.5);
  \coordinate (A4) at (0,1.5);
  \coordinate (A5) at (2,0);
  \coordinate (A6) at (2,1);

  \draw[thick] (A2) -- (A4) -- (A1) -- (A6) -- (A2);
  \draw[thick] (A2) -- (A3) -- (A4) -- (A5) -- (A1) -- (A6) -- (A3);
  
  \draw[dashed, thick] (A4) -- (A6) -- (A5) -- (A3) -- (A1) -- (A2);
  
  \node[draw,circle,inner sep=1.4pt,fill] at (A1) {};
  \node[draw,circle,inner sep=1.4pt,fill] at (A2) {};
  \node[draw,circle,inner sep=1.4pt,fill] at (A3) {};
  \node[draw,circle,inner sep=1.4pt,fill] at (A4) {};
  \node[draw,circle,inner sep=1.4pt,fill] at (A5) {};
  \node[draw,circle,inner sep=1.4pt,fill] at (A6) {};
\end{scope}

\begin{scope}[shift={(6,0)},scale=1]
  \coordinate (A1) at (-2,0);
  \coordinate (A2) at (-2,1);
  \coordinate (A3) at (-0.5,0);
  \coordinate (A4) at (0,1);
  \coordinate (A5) at (0.5,0);
  \coordinate (A6) at (2,0);
  \coordinate (A7) at (2,1);
	
  \draw[thick] (A3) -- (A4) -- (A5) -- (A3);
  \draw[thick] (A1) -- (A2) -- (A4);
  \draw[thick] (A2) -- (A5);
  \draw[thick] (A1) edge [bend right] (A5);
  \draw[thick] (A3) edge [bend right] (A6);
  \draw[thick] (A6) -- (A7) -- (A4);
  \draw[thick] (A7) -- (A3);
  
  \draw[dashed, thick] (A2) -- (A3) -- (A1) -- (A4) -- (A6) -- (A5) -- (A7);
  
  \node[draw,circle,inner sep=1.4pt,fill] at (A1) {};
  \node[draw,circle,inner sep=1.4pt,fill] at (A2) {};
  \node[draw,circle,inner sep=1.4pt,fill] at (A3) {};
  \node[draw,circle,inner sep=1.4pt,fill] at (A4) {};
  \node[draw,circle,inner sep=1.4pt,fill] at (A5) {};
  \node[draw,circle,inner sep=1.4pt,fill] at (A6) {};
  \node[draw,circle,inner sep=1.4pt,fill] at (A7) {};
\end{scope}
\end{tikzpicture}
\caption{Decompositions of $K_6-e$ (left) and  $K_5 \cup_{K_3} K_5$ (right) into a spanning $(2,3)$-tight subgraph (not dashed) and a spanning tree (dashed).}
\label{fig:k5glued}
\end{figure}
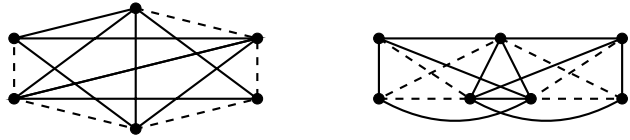

\begin{theorem}\label{t:braced}
Let $X$ be a generic normed plane and let $G$ be a simple graph. 
Suppose one of the following conditions holds.
 \begin{enumerate}[(a)]
 \item $X$ is isometrically isomorphic to $\mathbb{R}^2$ and $G$ is formed from a triangulation of the 2-sphere plus two extra edges.
 \item $X$ is not isometrically isomorphic to $\mathbb{R}^2$ and $G$ is formed from a triangulation of the 2-sphere plus three extra edges.  
 \end{enumerate}
 Then $G$ is minimally rigid in the cylindrical normed space $X \oplus_\infty \mathbb{R}$.
\end{theorem}

\begin{proof}
Suppose $(a)$ holds. Note that both $K_6-e$ and $K_5 \cup_{K_3} K_5$ can be expressed as an edge-disjoint union of a spanning $(2,3)$-tight subgraph and a spanning tree  (see Figure \ref{fig:k5glued}). Thus, by Theorem \ref{mainthm}(a), $K_6-e$ and $K_5 \cup_{K_3} K_5$ are minimally rigid in $X \oplus_\infty \mathbb{R}$.
    The result now follows  from  Proposition \ref{p:2dvertsplit} and Proposition \ref{p:CruKasKitSch}.

If $(b)$ holds then $G$ is $(3,3)$-tight and so the result follows from Theorem \ref{mainthm}(b).
\end{proof}

An immediate corollary of Theorem \ref{t:braced} is that all planar graphs are independent in the cylindrical normed space $X \oplus_\infty \mathbb{R}$ whenever $X$ is a generic normed plane.
An analogous statement holds true for 3-dimensional normed spaces that are both strictly convex and smooth;
see \cite[Theorem 6.3]{DewKitNix}. Note that cylindrical normed spaces are neither strictly convex nor smooth.

\subsection{Triangulations of the projective plane or the torus}

In the following, we denote by  $K_7 - K_3$ the complete graph on seven vertices with three edges that form a cycle removed.

\begin{proposition}[Barnette \cite{Barnette}]\label{p:VsplitPrjPlanar}
    Let $G$ be a triangulation of the projective plane.
    Then $G$ can be constructed from either $K_6$ or $K_7 - K_3$ by a sequence of 3-dimensional vertex splitting operations.
\end{proposition}


\begin{theorem}\label{t:prjplanar}
    Let $X$ be a generic normed plane and let $G$ be a triangulation of the projective plane.
\begin{enumerate}[(a)]
\item If $X$ is isometrically isomorphic to $\mathbb{R}^2$ then $G$ is rigid, but not minimally rigid, in the cylindrical normed space $X \oplus_\infty \mathbb{R}$.
\item If $X$ is not isometrically isomorphic to $\mathbb{R}^2$ then  $G$ is minimally rigid in the cylindrical normed space $X \oplus_\infty \mathbb{R}$.
\end{enumerate}
\end{theorem}

\begin{proof}
$(a)$
As noted in the proof of Theorem \ref{t:braced}(a), $K_6 -e$ and $K_5 \cup_{K_3} K_5$ are minimally rigid in $X \oplus_\infty \mathbb{R}$.
Since $K_6$ contains a spanning copy of $K_6 -e$ and $K_7 - K_3$ contains a spanning copy of $K_5 \cup_{K_3} K_5$,
both $K_6$ and $K_7-K_3$ are rigid in $X \oplus_\infty \mathbb{R}$.
The result now follows from Proposition \ref{p:2dvertsplit} and Proposition \ref{p:VsplitPrjPlanar}.

$(b)$
Note that $K_6$ and $K_7 - K_3$ are both expressible as an edge disjoint union of three spanning trees (see Figure \ref{fig:K7-K3}). Thus, by Theorem \ref{mainthm}(b), $K_6$ and $K_7 - K_3$ are minimally rigid in $X\oplus_\infty \mathbb{R}$.  The result now follows  from Proposition \ref{p:2dvertsplit} and Proposition \ref{p:VsplitPrjPlanar}.
\end{proof}

\begin{figure}[ht]
 \centering
\begin{tikzpicture}
\begin{scope}[shift={(0,0)},scale=0.8]
  \coordinate (A1) at (-2,0.25);
  \coordinate (A2) at (-2,1.75);
  \coordinate (A3) at (0,-0.5);
  \coordinate (A4) at (0,2.5);
  \coordinate (A5) at (2,0.25);
  \coordinate (A6) at (2,1.75);

  \draw[very thick] (A3) -- (A4) -- (A2) -- (A6) -- (A1) -- (A5);
  \draw[very thick,gray] (A1) -- (A4) -- (A5) -- (A2) -- (A3) -- (A6);
  
  \draw[dashed, thick] (A4) -- (A6) -- (A5) -- (A3) -- (A1) -- (A2);
  
  \node[draw,circle,inner sep=1.4pt,fill] at (A1) {};
  \node[draw,circle,inner sep=1.4pt,fill] at (A2) {};
  \node[draw,circle,inner sep=1.4pt,fill] at (A3) {};
  \node[draw,circle,inner sep=1.4pt,fill] at (A4) {};
  \node[draw,circle,inner sep=1.4pt,fill] at (A5) {};
  \node[draw,circle,inner sep=1.4pt,fill] at (A6) {};
\end{scope}

\begin{scope}[shift={(6,0)},scale=0.8]
  \coordinate (A1) at (-2,0.25);
  \coordinate (A2) at (-2,1.75);
  \coordinate (A3) at (-0.75,-0.5);
  \coordinate (A4) at (0,2.5);
  \coordinate (A5) at (0.75,-0.5);
  \coordinate (A6) at (2,0.25);
  \coordinate (A7) at (2,1.75);
	
  \draw[very thick] (A1) -- (A2) -- (A4);
  \draw[very thick] (A2) -- (A5);
  \draw[very thick, gray] (A1) -- (A5);
  \draw[very thick, gray] (A3) -- (A6);
  \draw[very thick, gray] (A4) -- (A7);
  \draw[very thick, gray] (A1) -- (A7) -- (A2) -- (A6);
  \draw[very thick] (A1) -- (A6) -- (A7);
  \draw[very thick] (A7) -- (A3);
  
  \draw[dashed, thick] (A2) -- (A3) -- (A1) -- (A4) -- (A6) -- (A5) -- (A7);
  
  \node[draw,circle,inner sep=1.4pt,fill] at (A1) {};
  \node[draw,circle,inner sep=1.4pt,fill] at (A2) {};
  \node[draw,circle,inner sep=1.4pt,fill] at (A3) {};
  \node[draw,circle,inner sep=1.4pt,fill] at (A4) {};
  \node[draw,circle,inner sep=1.4pt,fill] at (A5) {};
  \node[draw,circle,inner sep=1.4pt,fill] at (A6) {};
  \node[draw,circle,inner sep=1.4pt,fill] at (A7) {};
\end{scope}
\end{tikzpicture}
\caption{Decompositions of $K_6$ (left) and  $K_7-K_3$ (right) into three edge-disjoint spanning trees.}
\label{fig:K7-K3}
\end{figure}
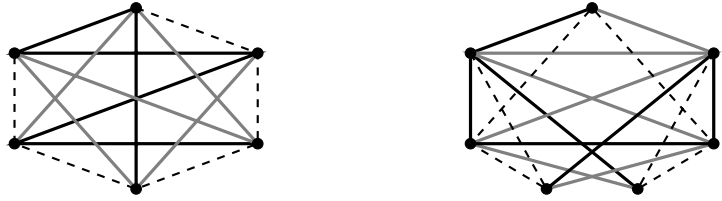

\begin{proposition}[Lavrenchenko \cite{lav}]\label{t:VsplitTorus}
    Let $G$ be a triangulation of the torus.
    Then $G$ can be constructed from either $K_7$ or one of the graphs pictured in Figure \ref{fig:torus} (see Appendix \ref{sec:torus}) by a sequence of 3-dimensional vertex splitting operations.
\end{proposition}

\begin{theorem}\label{t:torus}
     Let $X$ be a generic normed plane and let $G$ be a triangulation of the torus.
Then $G$ is rigid, but not minimally rigid, in the cylindrical normed space $X \oplus_\infty \mathbb{R}$.
\end{theorem}

\begin{proof}
  Suppose $X$ is isometrically isomorphic to $\mathbb{R}^2$.
  As $K_7$ contains a spanning copy of $K_5 \cup_{K_3} K_5$,
  it is rigid in $X \oplus_\infty \mathbb{R}$ by Theorem \ref{t:braced}.
  Likewise, any graph featured in Figure \ref{fig:torus} can be decomposed into a spanning tree and a graph that is rigid in $X$,
  and hence is rigid in $X \oplus_\infty \mathbb{R}$ by Theorem \ref{mainthm}(a).
  It now follows that $G$ is rigid in $X \oplus_\infty \mathbb{R}$ by Proposition \ref{p:2dvertsplit} and Theorem \ref{t:VsplitTorus}.

  Suppose $X$ is not isometrically isomorphic to $\mathbb{R}^2$.
  By a theorem of Kundu \cite{kundu}, $G$ contains three edge-disjoint spanning trees and so, by Theorem \ref{mainthm}(b), $G$ is rigid in $X\oplus_\infty \mathbb{R}$. 
\end{proof}

\subsection{Connectivity criteria}

The following corollary gives an affirmative answer to Conjecture 59(b) in \cite{matrixnorm}. 

\begin{corollary}\label{c:8econn}
    Let $G=(V,E)$ be a graph with $|V| \geq 5$ such that:
    \begin{enumerate}[(i)]
        \item $G$ is 8-edge-connected,
        \item $G-\{u\}$ is 6-edge-connected for all $u \in V$,
        \item $G-\{u,v\}$ is 4-edge-connected for all $u,v \in V$, and
        \item $G-\{u,v,w\}$ is 2-edge-connected for all $u,v,w \in V$.
    \end{enumerate}
    Then $G$ is rigid in the cylindrical normed space $\mathbb{R}^2 \oplus_\infty \mathbb{R}$. In particular, every $8$-connected graph is  rigid in $\mathbb{R}^2 \oplus_\infty \mathbb{R}$.
\end{corollary}

\proof
By \cite[Theorem 4]{packing}, $G$ is an edge-disjoint union of a spanning $(2,3)$-tight subgraph and a spanning tree. Thus the result follows from Theorem \ref{mainthm}(a).
\endproof



We see that 8-connected is best possible in Corollary \ref{c:8econn} with the following example.
Here we use a similar method to that employed in \cite{lovasz} to prove that 5-connectivity is not sufficient for rigidity in the Euclidean plane.

\begin{example}
    Choose any 7-connected 7-regular graph $G'=(V',E')$ with $|V'| \geq 10$ (for example, the complete bipartite graph $K_{7,7}$).
    We now construct the graph $G=(V,E)$ from $G'$  by replacing each vertex of $G'$ with a copy of $K_7$ and then sharing out the edges equally so that $G$ is 7-regular. More precisely, to construct $G$ we
    (i) replace every vertex $v \in V'$ with 7 copies $v_1,\ldots,v_7$,
    (ii) replace every edge $vw\in E'$ with an edge $v_iw_j$ such that each vertex $v_i$ is adjacent to exactly one vertex that did not originally stem from $v$,
    and
    (iii) for every vertex $v \in V'$, add all edges $v_iv_j$ for every pair $1 \leq i < j \leq 7$.
    Since $G'$ is 7-connected and 7-regular, so too is $G$.
    By abuse of notation we label the edges of $G$ that came from the graph $G'$ by $E'$ also.
    
    Let $T$ be any spanning tree of $G$.
    Note that $T$ must include at least $|V'|-1$ edges in $E'$.
    For any graph $H=(U,F)$,
    define $\rank (H)$ to be the cardinality of the largest edge set $F' \subset F$ such that $H'=(U,F')$ is $(2,3)$-sparse;
    by Theorem \ref{t:2d}(a), $H$ is rigid in $\mathbb{R}^2$ if and only if $\rank(H) = 2|U|-3$ (or $|U|=1$).
    Using a result of Lov\'{a}sz and Yemini \cite[Theorem 1]{lovasz},
    we compute that
    \begin{align*}
        \rank(G-T) &\leq |E'\setminus E(T)| + \sum_{v \in V'} \rank(K_7) \\
        &\leq \frac{7|V'|}{2} - (|V'| - 1) + (2|V| - 3|V'|) \\
        &= 2|V|- \frac{|V'|}{2} + 1 \\
        &< 2|V|-3.
    \end{align*}
    Hence $G-T$ is not rigid in $\mathbb{R}^2$.
    As this
    holds for any spanning tree of $G$,
    it follows from Theorem \ref{mainthm}(a) that $G$ has no well-positioned infinitesimally rigid placement in $\mathbb{R}^2 \oplus_\infty \mathbb{R}$.
\end{example}

The following corollary gives an affirmative answer to Conjecture 59(c) in \cite{matrixnorm}. 

\begin{corollary}
    Let $G=(V,E)$ be a graph such that:
    \begin{enumerate}[(i)]
        \item $G$ is 6-connected,
        \item Every vertex-induced subgraph obtained by deleting up to 7 vertices from $G$ has at most one connected component with a non-empty edge set.
    \end{enumerate}
    Then $G$ is rigid in the cylindrical normed space $\mathbb{R}^2 \oplus_\infty \mathbb{R}$.
\end{corollary}

\proof
By \cite[Corollary 1.10]{gu}, $G$ is an edge-disjoint union of a spanning $(2,3)$-tight subgraph and a spanning tree. Thus the result follows from Theorem \ref{mainthm}(a).
\endproof


\begin{corollary}\label{c:6edgeconn}
    Let $X$ be a generic normed plane which is not isometrically isomorphic to $\mathbb{R}^2$ and let $G=(V,E)$ be a 6-edge-connected graph.
    Then $G$ is rigid in the cylindrical normed space $X \oplus_\infty \mathbb{R}$.
\end{corollary}

\begin{proof}
By \cite[Theorem 2]{kundu}, $G$ contains three  edge-disjoint spanning trees. Thus the result follows from Theorem \ref{mainthm}(b).
\end{proof}

Note that the 6-edge-connectivity condition in Corollary \ref{c:6edgeconn} is best possible. For example, the graph in Figure \ref{fig:34eg} has an edge-connectivity of 5 but only has 56 edges, one short of the required number given by Theorem \ref{mainthm}(b).

\section{Related work}
\label{S:Related}

\subsection{Algorithms for rigidity in cylindrical normed spaces}
\label{sec:matroidalg}

A \emph{matroid} is a pair $\mathcal{M} = (E,\mathcal{I})$ where $E$ is a finite set and $\mathcal{I}$ is a collection of subsets of $E$ such that the following three properties are satisfied:
\begin{enumerate}
    \item[(I1)] $\emptyset \in \mathcal{I}$.
    \item[(I2)] If $I \in \mathcal{I}$ and $J \subset I$, then $J \in \mathcal{I}$.
    \item[(I3)] If $I,J \in \mathcal{I}$ and $|J| < |I|$, then there exists $e \in I \setminus J$ such that $J \cup \{e\} \in \mathcal{I}$.
\end{enumerate}
The set $E$ is called the ground set and the elements of $\mathcal{I}$ are said to be independent.
 The \emph{row matroid} of a real matrix $M$ is the matroid $\mathcal{M} = (E, \mathcal{I})$ where $E$ is the set of row labels for $M$ and a subset $I \subseteq E$ is independent in $\mathcal{M}$ if and only if the corresponding rows of $M$ are linearly independent over $\mathbb{R}$.
The {\em graphic matroid} for a graph $G=(V,E)$ is the matroid $M(G)=(E,\mathcal{I})$ where the independent sets correspond to forests in $G$. Note that $M(G)$ is the row matroid for the incidence matrix of the graph $G$. 

Let $K_V =(V,K(V))$ be the complete graph on a finite vertex set $V$ and let $X$ be a finite dimensional real normed linear space. 
A placement $p\in X^V$  is \emph{completely regular} if the framework $(K_V, p)$, and every subframework of $(K_V, p)$, is regular.
Given a completely regular placement $p\in X^V$, we define the \emph{rigidity matroid} $R_X(V)$ to be the row matroid of the rigidity matrix $R(K_V,p)$ (see Section \ref{s:wp}).
Note that every completely regular placement $p\in X^V$ will generate the same rigidity matroid.
With this terminology,
a graph $G=(V,E)$ is independent in $X$ if and only if $E$ is independent in the rigidity matroid $R_X(V)$.
If $X$ is a generic space then  the set of completely regular placements of $K_V$ is a dense subset of $X^V$. This follows since the set of regular placements of a graph $G=(V,E)$ in $X$ is an open subset of the set of well-positioned placements of $K_V$ in $X$, which in turn is a dense subset of $X^V$; see \cite[Section 4.1]{dew1} for more details. 


Given two matroids $\mathcal{M}_1 = (E, \mathcal{I}_1)$ and $\mathcal{M}_2 = (E, \mathcal{I}_2)$ with the same ground set $E$,
define their \emph{matroid union} to be the matroid $\mathcal{M}_1 \vee \mathcal{M}_2 := (E,\mathcal{I})$ where,
\begin{align*}
    \mathcal{I} := \{ I_1 \cup I_2 : I_1 \in \mathcal{I}_1, ~ I_2 \in \mathcal{I}_2, ~ I_1 \cap I_2 = \emptyset \}.
\end{align*}

\begin{lemma}[{\cite[Lemma 7.6.14(1)]{Bry}}]\label{l:matroidunion}
    Let $\mathcal{M}_1 = (E, \mathcal{I}_1)$ and $\mathcal{M}_2 = (E, \mathcal{I}_2)$ be two matroids with the same ground set.
    Suppose that each $\mathcal{M}_i$ is the row matroid of a real $|E| \times n_i$ matrix $M_i$.
    Let $D$ be a real $|E| \times |E|$ diagonal matrix.
    If the diagonal entries of $D$ are algebraically independent over $\mathbb{Q}(M_1)$ (the smallest field containing the rational numbers and the matrix entries of $M_1$),
    then the matroid union $\mathcal{M}_1 \vee \mathcal{M}_2$ is the row matroid of the matrix $[ M_1 ~ ~ D M_2 ]$.
\end{lemma}

The following matroidal characterisation of independence in cylindrical normed spaces is an immediate consequence of Theorem \ref{graphthm}.

\begin{corollary}\label{c:matroid}
    Let $X$ be a generic normed space of dimension at least $2$.
    A graph $G=(V,E)$ is independent in $X \oplus_\infty \mathbb{R}$ if and only if $E$ is independent in $R_X(V) \vee M(G)$.
\end{corollary}

Using Lemma \ref{l:matroidunion} and Corollary \ref{c:matroid},
we now present an algorithm for determining whether a graph $G=(V,E)$ is independent in a cylindrical normed space $X \oplus_\infty \mathbb{R}$,
where $X$ is a generic normed space of dimension at least $2$.
\begin{enumerate}[(1)]
\item First, choose a placement $p$ of $G$ in $X$ and use this to construct the rigidity matrix $R(G,p)$ of the framework $(G,p)$ in $X$.
The placement $p$ should be chosen such that $R_X(V)$ is the row matroid of $(K_V,p)$ (where $K_V$ is the complete graph with vertex set $V$);
since $X$ is generic, a random choice of placement will almost surely satisfy this property.
\item For each edge $e \in E$ choose (i) an initial vertex (which we label $i(e)$) and (ii) a random real number $b_e$.
With this,
define $B(G)$ to be the $|E| \times |V|$ matrix with entries
\begin{align*}
    B(G)_{e,v} :=
    \begin{cases}
        b_e &\text{if $e =vw$ and $i(e) = v$},\\
        -b_e &\text{if $e =vw$ and $i(e) = w$},\\
        0 &\text{otherwise}.
    \end{cases}
\end{align*}
Importantly, the matrix $B(G)$ is formed from the directed incidence matrix of $G$ (with our arbitrary choice of edge directions) by multiplying each row $e$ by $b_e$.
\item Let  $M = [R(G,p) ~ B(G)]$ be the $|E| \times (d+1)|V|$ matrix with matrix $R(G,p)$ forming the left columns and matrix $B(G)$ forming the right columns. It follows from Lemma \ref{l:matroidunion} that $G$ is independent in $X \oplus_\infty \mathbb{R}$  if $\rank M = |E|$, and $G$ is almost surely dependent in $X \oplus_\infty \mathbb{R}$ if $\rank M < |E|$.
\end{enumerate}

If the computational speed of determining the support functionals of the chosen framework is polynomial (for example, if $X$ is an $\ell_p$ space),
then the above algorithm  will run in polynomial time.
The algorithm is, however, not deterministic.
This can, in some cases, be solved by the use of \emph{Edmond's algorithm} \cite{edmond}; see \cite[Algorithm 11.1]{matroidunion} for a description of the general algorithm.
The algorithm takes any pair of matroids $\mathcal{M}_1 = (E, \mathcal{I}_1), \mathcal{M}_2 = (E, \mathcal{I}_2)$ with maximal independent subsets $B_1,B_2$ respectively and returns a maximal independent subset of $ \mathcal{M}_1 \vee \mathcal{M}_2$.
See Algorithm \ref{alg1} (Appendix \ref{sec:alg}) for the specific application of Edmond's algorithm to a graph $G=(V,E)$ and the cylindrical normed space $X \oplus_\infty \mathbb{R}$, with $X$ a generic normed space.

If independence in the rigidity matroid $R_X(V)$ can be determined by a deterministic algorithm then Algorithm \ref{alg1} will also be deterministic.
Furthermore,
if independence in $R_X(V)$ can be checked in polynomial time then Algorithm \ref{alg1} will run in polynomial time \cite[Section 11.3.3]{matroidunion}.
For example,
if $X$ is a generic rigidity space with dimension 2,
then a direct consequence of Theorem \ref{t:2d} tells us we can check whether an edge set is independent in $R_X(V)$ with the $(2,k)$ pebble game algorithm (see \cite{LeeStreinu}) with $k=3$ if $X$ is isometrically isomorphic to $\mathbb{R}^2$ and $k=2$ otherwise.
Since the $(2,k)$ pebble game algorithm runs in polynomial time,
Algorithm \ref{alg1} is a polynomial-time deterministic algorithm whenever $X$ is a generic space with dimension 2.

\subsection{Graph rigidity in conical normed spaces}
\label{s:conical}
A normed linear space is said to be {\em conical} if it is isometrically isomorphic to a direct sum $Z =X\oplus_1 \mathbb{R}$ where $X$ is a finite dimensional real  normed linear space and $Z$ is endowed with the norm $\|(x,y)\|_1 :=  \|x\|_X+|y|$.
Recall that the dual space of a conical normed space $X\oplus_1\mathbb{R}$  is isometrically isomorphic to the cylindrical normed space $X^*\oplus_\infty \mathbb{R}$.



\begin{lemma}
\label{l:diff2}
Let $Z=X\oplus_1 \mathbb{R}$  be a conical normed space and let $z=(x,y)\in Z$. 
\begin{enumerate}[(i)]
\item $z$ is a smooth point in $Z$ if and only if $x$ is a smooth point in $X$ and $y$ is non-zero. 
\item If $z$ is a smooth point in $Z$ then the linear functional,
\[
\psi:Z\to\mathbb{R},\quad
\psi(a,b) = 
\left\{
\begin{array}{ll}
\left(\frac{\varphi_x(a)}{\|x\|_X}+b\right) \|z\|_Z   &  \mbox{if } y > 0,\\[8pt]
\left(\frac{\varphi_x(a)}{\|x\|_X}-b\right) \|z\|_Z   &  \mbox{if } y < 0.
\end{array}
\right.
\]
is the unique support functional for $z$.
\end{enumerate}
\end{lemma}

\proof
$(i)$ For each  $u=(a,b)\in Z$, note that
$\|z+tu\|_Z = \|x+ta\|_X+|y+tb|$ for all $t$.

$(ii)$
By $(i)$, $y \neq 0$. 
Note that 
\[\psi(z) = \left(\frac{\varphi_x(x)}{\|x\|_X}+|y|\right)\|z\|_Z = (\|x\|_X+|y|)\|z\|_Z=\|z\|_Z^2.\]
Also, for each  $u=(a,b)\in Z$,  
\[|\psi(u)|\leq\left(\frac{|\varphi_x(a)|}{\|x\|_X}+|b|\right)\|z\|_Z\leq (\|a\|_X+|b|)\|z\|_Z = \|u\|_Z\|z\|_Z\]
Thus $\|\psi\|^*=\|z\|_Z$ and so $\psi=\varphi_z$ is the unique support functional for $z$.

\endproof

\begin{example}

Consider the conical normed space  $\ell_{q,1}^d=\ell_q^d\oplus_1 \mathbb{R}$ where $q\in[1,\infty)$ and $d\geq 1$. Let $z=(x,y)\in \ell_q^d\oplus_1 \mathbb{R}$ be a smooth point  with unique support functional $\varphi_z$ and write $x=(x_1,\ldots, x_d)$. 
Let $u=(a,b)\in\ell_q^d\oplus_1 \mathbb{R}$ and write $a=(a_1,\ldots,a_d)$.
If $y>0$ then, by Lemma \ref{l:diff2}(ii),
\[\frac{\varphi_z(u)}{\|z\|_Z}  =  \left(\sum_{i=1}^d \frac{\sgn(x_i)|x_i|^{q-1}}{\|x\|_q^{q-1}}a_i\right)+b= \begin{bmatrix}\frac{\sgn(x_1)|x_1|^{q-1}}{\|x\|_q^{q-1}}&\cdots&\frac{\sgn(x_d)|x_d|^{q-1}}{\|x\|_q^{q-1}}&1\end{bmatrix}\begin{bmatrix}a_1\\\vdots\\a_d\\b\end{bmatrix}\]
If $y<0$ then, by Lemma \ref{l:diff2}(ii),
\[
\frac{\varphi_z(u)}{\|z\|_Z}  =  \left(\sum_{i=1}^d \frac{\sgn(x_i)|x_i|^{q-1}}{\|x\|_q^{q-1}}a_i\right)-b= \begin{bmatrix}\frac{\sgn(x_1)|x_1|^{q-1}}{\|x\|_q^{q-1}}&\cdots&\frac{\sgn(x_d)|x_d|^{q-1}}{\|x\|_q^{q-1}}&-1\end{bmatrix}\begin{bmatrix}a_1\\\vdots\\a_d\\b\end{bmatrix}\]
The conical normed space  $\ell_{2,1}^2=\mathbb{R}^2\oplus_1 \mathbb{R}$ is of particular interest. Here, if $y>0$ then,
$\frac{\varphi_z(u)}{\|z\|_Z} = \begin{bmatrix}\frac{x_1}{\|x\|_2} &\frac{x_2}{\|x\|_2}&1\end{bmatrix}\begin{bmatrix}a_1\\a_2\\b\end{bmatrix}$
and if  $y<0$ then,
$\frac{\varphi_z(u)}{\|z\|_Z} = \begin{bmatrix}\frac{x_1}{\|x\|_2} &\frac{x_2}{\|x\|_2}&-1\end{bmatrix}\begin{bmatrix}a_1\\a_2\\b\end{bmatrix}$.
\end{example}

\begin{example}
Let $\mathcal{H}_\infty(2,\mathbb{R})$ denote the real linear space of $2\times 2$ real symmetric matrices endowed with the spectral norm and let $\mathcal{H}_\infty(2,\mathbb{C})$ denote the real linear space of $2\times 2$ complex hermitian  matrices also endowed with the spectral norm.
Recall that $\mathcal{H}_\infty(2,\mathbb{K})$ is isometrically isomorphic to  the dual of $\mathcal{H}_1(2,\mathbb{K})$ for $\mathbb{K}=\mathbb{R}$ or $\mathbb{C}$. As noted in Example \ref{ex:matrixnorms}, $\mathcal{H}_1(2,\mathbb{R})$ is a cylindrical normed space which is isometrically isomorphic to $\ell_{2,\infty}^3 = \mathbb{R}^2 \oplus_\infty \mathbb{R}$. Thus $\mathcal{H}_\infty(2,\mathbb{R})$ is a conical normed space as it  is isometrically isomorphic to  
$\ell_{2,1}^3 = \mathbb{R}^2 \oplus_1 \mathbb{R}$.
Similarly, $\mathcal{H}_\infty(2,\mathbb{C})$ is a conical normed space as it is isometrically isomorphic to  $\ell_{2,\infty}^4 = \mathbb{R}^3 \oplus_1 \mathbb{R}$.
\end{example}

Let $(G,p)$ be a  framework in a conical normed space $Z =X \oplus_1 \mathbb{R}$.
As before, we denote by $\pi_{X}$ and $\pi_{\mathbb{R}}$ the projections  from $Z$ onto $X$ and $\mathbb{R}$ respectively (see Equation \ref{eq:natproj}) and we define the projected frameworks $(G,p_{X})$ and $(G,p_{\mathbb{R}})$ accordingly.

\begin{lemma}\label{lem:wp2}
Let $(G,p)$ be a framework in a conical normed space $X\oplus_1 \mathbb{R}$.
The following statements are equivalent.
\begin{enumerate}[(i)]
    \item $(G,p)$ is well-positioned in $X\oplus_1 \mathbb{R}$.
    \item The projected framework $(G,p_X)$ is well-positioned in $X$ and $\pi_\mathbb{R}(p_v)\not=\pi_\mathbb{R}(p_w)$ for each edge $vw\in E$.
\end{enumerate}
\end{lemma}

\proof
Apply \cite[Proposition 6]{kit-sch} and Lemma \ref{l:diff2}.
\endproof

Given a fixed orientation $\delta$ of the edges of a graph $G=(V,E)$,
we denote by $I(G,\delta)$ the directed incidence matrix,
i.e., the $|E| \times |V|$ matrix with entries $a_{e,v}$ where
\begin{align*}
    a_{e,v} :=
    \begin{cases}
        1 &\text{if } e = vw \text{ and $e$ is directed from $v$ to $w$ with respect to $\delta$},\\
        -1 &\text{if } e = vw \text{ and $e$ is directed from $w$ to $v$ with respect to $\delta$},\\
        0 &\text{otherwise.}
    \end{cases}
\end{align*}
Given a framework $(G,q)$ in a normed space $X$ and edge orientation $\delta$ on $G$  
we define the matrices,
    \begin{align*}
        D(G,q,\delta) := D(G,q)I(G,\delta), \qquad       
        M(G,q,\delta) := \Big[ R(G,q) \quad D(G,q,\delta) \Big].
    \end{align*}
(Recall that here $D(G,q)$ is the diagonal matrix with rows and columns indexed by $E$ and $(vw,vw)$-entry $\|q_v-q_w\|_X$ for each edge $vw\in E$.)
We will require the following result of Cros, Amblard, Prieur and Da Rocha (\cite{CAPR}).

\begin{theorem}[{\cite[Theorem 4.2]{CAPR}}]
\label{t:capr}
Let $(G,q)$ be a framework in $\mathbb{R}^d$ and let $\delta$ be an edge orientation on $G$.
    If the coordinates of $q$ are algebraically independent over $\mathbb{Q}$ and $|V|\geq d+1$ then the following statements are equivalent.
\begin{enumerate}[(i)]
\item $\rank M(G,q,\delta) = (d+1)|V| - \binom{d+1}{2} - 1$.
\item $G$ contains edge-disjoint spanning subgraphs $H$ and $T$ such that $H$ is rigid in $\mathbb{R}^d$ and $T$ is a tree.
\end{enumerate}
\end{theorem}

If $(G,p)$ is a well-positioned framework in a conical normed space $X\oplus_1\mathbb{R}$ then we define $\delta_p$ to be the edge orientation on $G$ where an edge $vw$ is directed from $v$ to $w$ if and only if $\pi_{\mathbb{R}}(p_v)>\pi_{\mathbb{R}}(p_w)$.
    
 \begin{lemma}\label{l:l1prod}
     Let $(G,p)$ be a well-positioned framework in a conical normed space $X \oplus_1 \mathbb{R}$.
     Then $\rank R(G,p) = \rank M(G,p_X,\delta_p)$.
 \end{lemma}

 \proof
 The result follows from Lemma \ref{l:diff2}. 
 \endproof



We are now ready to prove an equivalence between rigidity in the cylindrical normed space $X\oplus_\infty\mathbb{R}$ and rigidity in the conical normed space $X\oplus_1\mathbb{R}$ when $X = \mathbb{R}^d$.

\begin{theorem}\label{t:l1prod}
    For any graph $G=(V,E)$ with $|V|\geq d+1$,
    the following are equivalent.
    \begin{enumerate}[(i)]
        \item $G$ is minimally rigid in the conical space $\mathbb{R}^d \oplus_1 \mathbb{R}$.
        \item $G$ is minimally rigid in the cylindrical space $\mathbb{R}^d \oplus_\infty \mathbb{R}$.
        \item $G$ is an edge-disjoint union of spanning subgraphs $H$ and $T$, where $H$ is minimally rigid in $\mathbb{R}^d$ and $T$ is a tree.
    \end{enumerate}
\end{theorem}

\begin{proof}
$(i)\Leftrightarrow (iii)$
    Fix a placement $p$ of $G$ in $\mathbb{R}^d \oplus_1 \mathbb{R}$ such that the coordinates of $p$ are algebraically independent over $\mathbb{Q}$.
    Then $(G,p)$ is well-positioned, and by Lemma \ref{l:l1prod}, $\rank R(G,p) = \rank M(G,p_{\mathbb{R}^d},\delta_p)$.
    Note  that $\dim \mathcal{T}(G,p) = \dim \mathcal{T}(\mathbb{R}^d \oplus_1 \mathbb{R}) = \frac{d(d+1)}{2}+1$. 
    Hence the result follows by Theorem \ref{t:capr}.

 $(ii)\Leftrightarrow (iii)$   This is Theorem \ref{graphthm}.
\end{proof}

Using Theorem \ref{t:l1prod} and the results of Section \ref{Sec:Rigidity}, we now characterise minimally rigid graphs in the 4-dimensional cylindrical normed space $(\mathbb{R}^2 \oplus_1 \mathbb{R}) \oplus_\infty \mathbb{R}$.

\begin{lemma}
\label{l:generic}
    $\mathbb{R}^d \oplus_1 \mathbb{R}$ is a generic space.    
\end{lemma}

\proof
    It suffices for us to prove that for any graph $G=(V,E)$ that is independent in $\mathbb{R}^d \oplus_1 \mathbb{R}$,
    the set of regular placements of $G$ in $\mathbb{R}^d \oplus_1 \mathbb{R}$ is a dense subset of $(\mathbb{R}^d \oplus_1 \mathbb{R})^V$.
    Fix a placement $p$ of $G$ in $\mathbb{R}^d \oplus_1 \mathbb{R}$ such that the coordinates of $p$ are algebraically independent over $\mathbb{Q}$.
    Note that $(G,p)$ is well-positioned, and by Lemma \ref{l:l1prod}, $\rank R(G,p) = \rank M(G,p_{\mathbb{R}^d},\delta_p)$.
    
    We now prove that $(G,p)$ is independent and hence regular.
    First suppose that $|V| \leq d$.
    As the graph $G$ is independent in $\mathbb{R}^d$,
    our choice of $p$ implies that $\rank R(G,p_{\mathbb{R}^d}) = |E|$.
    From this we see that
    \begin{align*}
        |E| \geq \rank R(G,p) = \rank M(G,p_{\mathbb{R}^d},\delta_p) \geq \rank R(G,p_{\mathbb{R}^d})  = |E|,
    \end{align*}
    and so $(G,p)$ is independent.    
    Now suppose that $|V| \geq d+1$.
    If $G$ is rigid (and hence minimally rigid) in $\mathbb{R}^d \oplus_1 \mathbb{R}$, then it follows from Theorem \ref{graphthm}, Theorem \ref{t:capr} and Lemma \ref{l:l1prod} that 
    $$\rank R(G,p) = \rank M(G,p_{\mathbb{R}^d},\delta_p)=(d+1)|V| - \binom{d+1}{2} - 1=|E|.$$
    If $G$ is not rigid in $\mathbb{R}^d \oplus_1 \mathbb{R}$ then choose a graph $G' = (V',E')$ that is minimally rigid in $\mathbb{R}^d \oplus_1 \mathbb{R}$ and contains $G$ (this being possible due to Theorem \ref{t:l1prod}).
    Extend the placement $p$ to a placement $p'$ of $G'$ by setting $p'_v = p_v$ for all $v \in V$ and choosing the positions of the remaining vertices such that the coordinates of $p'$ are also algebraically independent over $\mathbb{Q}$.
    As $G'$ is rigid in $\mathbb{R}^d \oplus_1 \mathbb{R}$,
    our previous argument applied to the framework $(G',p')$ implies that $\rank R(G',p') = |E'|$,
    i.e., the rows of $R(G',p')$ are linearly independent.
    Hence the rows of the submatrix $R(G,p)$ of $R(G',p')$ are linearly independent and $\rank R(G,p) = |E|$ as required.
    
    The result now follows from the observation that the set of placements with coordinates that are algebraically independent over $\mathbb{Q}$ forms a dense subset of $(\mathbb{R}^d \oplus_1 \mathbb{R})^V$.
\endproof

\begin{theorem}\label{t:fullydefined}
    Let $Z = (\mathbb{R}^2 \oplus_1 \mathbb{R}) \oplus_\infty \mathbb{R}$
    be the cylindrical normed space with norm
    \begin{align*}
        \|(x_1,x_2,x_3,x_4)\|_Z := \max \left\{ \sqrt{x_1^2 + x_2^2} + |x_3| , ~ |x_4| \right\}.
    \end{align*}
    Then the following are equivalent for any graph $G=(V,E)$:
    \begin{enumerate}[(i)]
        \item $G$ is minimally rigid in $Z$.
        \item $G$ is an edge disjoint union of spanning subgraphs $H$, $T_1$ and $T_2$ where $H$ is $(2,3)$-tight and $T_1,T_2$ are trees.
    \end{enumerate}
\end{theorem}

\proof
Since, by Lemma \ref{l:generic}, the conical space $\mathbb{R}^2 \oplus_1 \mathbb{R}$ is generic we can apply Theorem \ref{graphthm} with $X=\mathbb{R}^2 \oplus_1 \mathbb{R}$.
The result now follows from Theorem \ref{t:l1prod} and Theorem \ref{t:2d}(a).
\endproof

Importantly,
Theorem \ref{t:fullydefined} is the first complete combinatorial characterisation of rigidity in a 4-dimensional normed space.
We conclude this section with the following sufficient connectivity criteria.

\begin{corollary}\label{c:10econn}
    Let $G=(V,E)$ be a graph with $|V| \geq 6$ such that:
    \begin{enumerate}[(i)]
        \item $G$ is 10-edge-connected,
        \item $G-\{u\}$ is 8-edge-connected for all $u \in V$,
        \item $G-\{u,v\}$ is 6-edge-connected for all $u,v \in V$,
        \item $G-\{u,v,w\}$ is 4-edge-connected for all $u,v,w \in V$, and
        \item $G-\{u,v,w,x\}$ is 2-edge-connected for all $u,v,w,x\in V$.
    \end{enumerate}
    Then $G$ is rigid in the normed space $(\mathbb{R}^2 \oplus_1 \mathbb{R}) \oplus_\infty \mathbb{R}$.
    In particular, every 10-connected graph is rigid in $(\mathbb{R}^2 \oplus_1 \mathbb{R}) \oplus_\infty \mathbb{R}$.
\end{corollary}

\begin{proof}
By \cite[Theorem 4]{packing}, $G$ contains an edge disjoint union of spanning subgraphs $H$, $T_1$ and $T_2$ where $H$ is $(2,3)$-tight and $T_1,T_2$ are trees. Thus the result follows from Theorem \ref{t:fullydefined}.
\end{proof}

\subsection*{Acknowledgement}
SD was supported by the Heilbronn Institute for Mathematical Research and the Austrian Science Fund (FWF): P31888.

\newpage

\begin{appendices}

\section{Irreducible triangulations of a torus}
\label{sec:torus}

We recall that a triangulation of a surface is said to be \emph{irreducible} if every edge lies in three or more cycles of length 3.
In Figure \ref{fig:torus} we picture all 20 irreducible triangulations of the torus with at least 8 vertices and a decomposition of each into a tree and a graph that is rigid in $\mathbb{R}^2$.

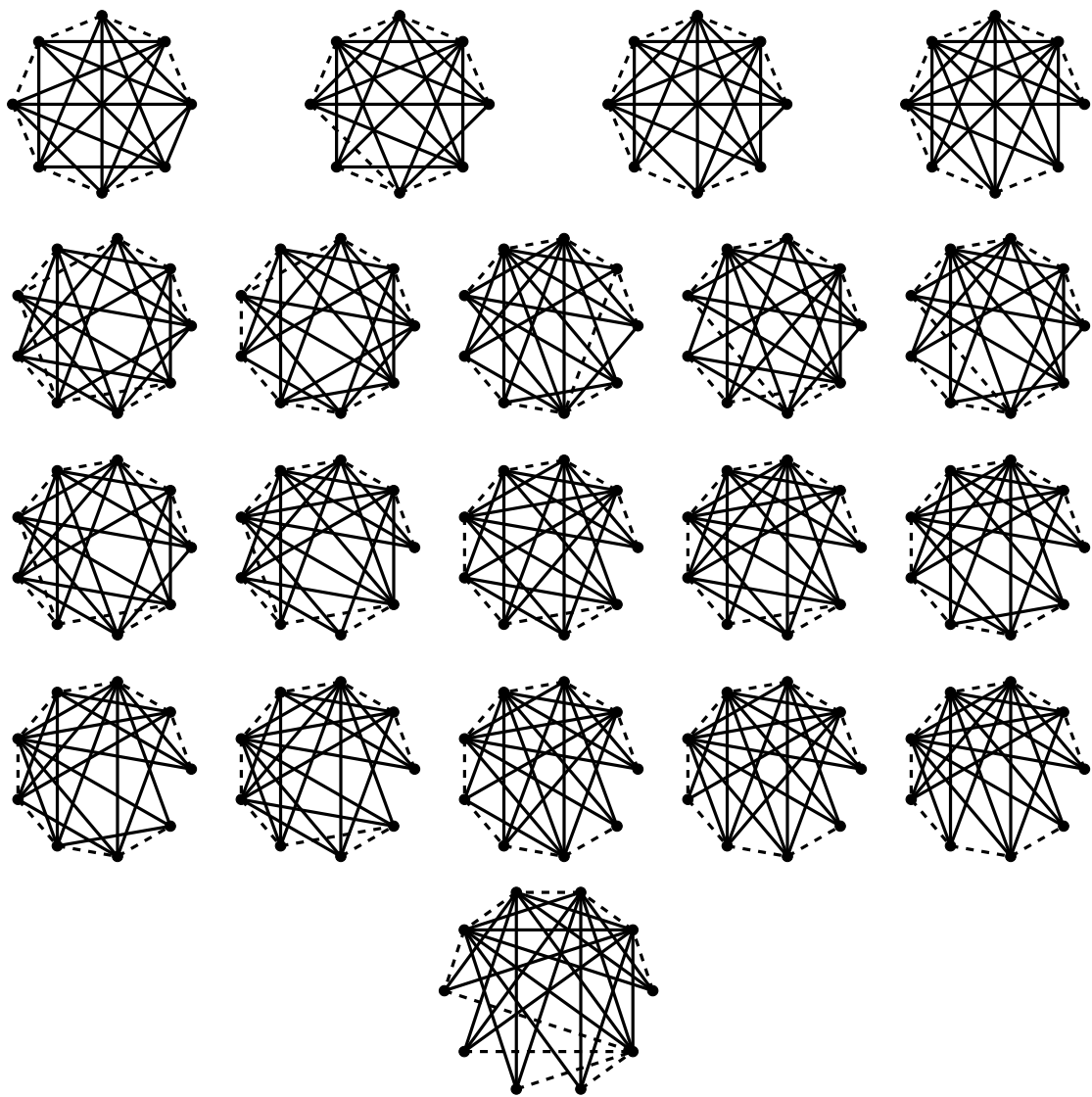
\begin{figure}[ht]
 \centering
\begin{tikzpicture}
\begin{scope}[shift={(0,3)},scale=1.2]
  \coordinate (A1) at (1,0);
  \coordinate (A2) at (0.707,0.707);
  \coordinate (A3) at (0,1);
  \coordinate (A4) at (-0.707,0.707);
  \coordinate (A5) at (-1,0);
  \coordinate (A6) at (-0.707,-0.707);
  \coordinate (A7) at (0,-1);
  \coordinate (A8) at (0.707,-0.707);

  \draw[very thick,dashed] (A1) -- (A2);
  \draw[very thick] (A1) -- (A3);
  \draw[very thick] (A1) -- (A4);
  \draw[very thick] (A1) -- (A5);
  \draw[very thick] (A1) -- (A7);
  \draw[very thick] (A1) -- (A8);
  
  \draw[very thick,dashed] (A2) -- (A3);
  \draw[very thick] (A2) -- (A4);
  \draw[very thick] (A2) -- (A5);
  \draw[very thick] (A2) -- (A6);
  \draw[very thick] (A2) -- (A7);
  
  \draw[very thick,dashed] (A3) -- (A4);
  \draw[very thick] (A3) -- (A6);
  \draw[very thick] (A3) -- (A7);
  \draw[very thick] (A3) -- (A8);
  
  \draw[very thick,dashed] (A4) -- (A5);
  \draw[very thick] (A4) -- (A6);
  \draw[very thick] (A4) -- (A8);
  
  \draw[very thick,dashed] (A5) -- (A6);
  \draw[very thick] (A5) -- (A7);
  \draw[very thick] (A5) -- (A8);
  
  \draw[very thick,dashed] (A6) -- (A7);
  \draw[very thick] (A6) -- (A8);
  
  \draw[very thick,dashed] (A7) -- (A8);
  
  \node[draw,circle,inner sep=1.4pt,fill] at (A1) {};
  \node[draw,circle,inner sep=1.4pt,fill] at (A2) {};
  \node[draw,circle,inner sep=1.4pt,fill] at (A3) {};
  \node[draw,circle,inner sep=1.4pt,fill] at (A4) {};
  \node[draw,circle,inner sep=1.4pt,fill] at (A5) {};
  \node[draw,circle,inner sep=1.4pt,fill] at (A6) {};
  \node[draw,circle,inner sep=1.4pt,fill] at (A7) {};
  \node[draw,circle,inner sep=1.4pt,fill] at (A8) {};
\end{scope}

\begin{scope}[shift={(4,3)},scale=1.2]
  \coordinate (A1) at (1,0);
  \coordinate (A2) at (0.707,0.707);
  \coordinate (A3) at (0,1);
  \coordinate (A4) at (-0.707,0.707);
  \coordinate (A5) at (-1,0);
  \coordinate (A6) at (-0.707,-0.707);
  \coordinate (A7) at (0,-1);
  \coordinate (A8) at (0.707,-0.707);

  \draw[very thick,dashed] (A1) -- (A2);
  \draw[very thick] (A1) -- (A3);
  \draw[very thick] (A1) -- (A4);
  \draw[very thick] (A1) -- (A5);
  \draw[very thick] (A1) -- (A7);
  
  \draw[very thick,dashed] (A2) -- (A3);
  \draw[very thick] (A2) -- (A4);
  \draw[very thick] (A2) -- (A5);
  \draw[very thick] (A2) -- (A6);
  \draw[very thick] (A2) -- (A7);
  \draw[very thick] (A2) -- (A8);
  
  \draw[very thick,dashed] (A3) -- (A4);
  \draw[very thick] (A3) -- (A5);
  \draw[very thick] (A3) -- (A6);
  \draw[very thick] (A3) -- (A8);
  
  \draw[very thick,dashed] (A4) -- (A5);
  \draw[very thick] (A4) -- (A6);
  \draw[very thick] (A4) -- (A7);
  \draw[very thick] (A4) -- (A8);
  
  \draw[very thick,dashed] (A5) -- (A7);
  \draw[very thick] (A5) -- (A8);
  
  \draw[very thick,dashed] (A6) -- (A7);
  \draw[very thick] (A6) -- (A8);
  
  \draw[very thick,dashed] (A7) -- (A8);
  
  \node[draw,circle,inner sep=1.4pt,fill] at (A1) {};
  \node[draw,circle,inner sep=1.4pt,fill] at (A2) {};
  \node[draw,circle,inner sep=1.4pt,fill] at (A3) {};
  \node[draw,circle,inner sep=1.4pt,fill] at (A4) {};
  \node[draw,circle,inner sep=1.4pt,fill] at (A5) {};
  \node[draw,circle,inner sep=1.4pt,fill] at (A6) {};
  \node[draw,circle,inner sep=1.4pt,fill] at (A7) {};
  \node[draw,circle,inner sep=1.4pt,fill] at (A8) {};
\end{scope}

\begin{scope}[shift={(8,3)},scale=1.2]
  \coordinate (A1) at (1,0);
  \coordinate (A2) at (0.707,0.707);
  \coordinate (A3) at (0,1);
  \coordinate (A4) at (-0.707,0.707);
  \coordinate (A5) at (-1,0);
  \coordinate (A6) at (-0.707,-0.707);
  \coordinate (A7) at (0,-1);
  \coordinate (A8) at (0.707,-0.707);

  \draw[very thick,dashed] (A1) -- (A2);
  \draw[very thick] (A1) -- (A3);
  \draw[very thick] (A1) -- (A4);
  \draw[very thick] (A1) -- (A5);
  \draw[very thick] (A1) -- (A7);
  
  \draw[very thick,dashed] (A2) -- (A3);
  \draw[very thick] (A2) -- (A4);
  \draw[very thick] (A2) -- (A5);
  \draw[very thick] (A2) -- (A6);
  \draw[very thick] (A2) -- (A7);
  \draw[very thick] (A2) -- (A8);
  
  \draw[very thick,dashed] (A3) -- (A4);
  \draw[very thick] (A3) -- (A5);
  \draw[very thick] (A3) -- (A6);
  \draw[very thick] (A3) -- (A7);
  \draw[very thick] (A3) -- (A8);
  
  \draw[very thick,dashed] (A4) -- (A5);
  \draw[very thick] (A4) -- (A6);
  \draw[very thick] (A4) -- (A8);
  
  \draw[very thick,dashed] (A5) -- (A6);
  \draw[very thick] (A5) -- (A7);
  \draw[very thick] (A5) -- (A8);
  
  \draw[very thick,dashed] (A6) -- (A7);
  
  \draw[very thick,dashed] (A7) -- (A8);
  
  \node[draw,circle,inner sep=1.4pt,fill] at (A1) {};
  \node[draw,circle,inner sep=1.4pt,fill] at (A2) {};
  \node[draw,circle,inner sep=1.4pt,fill] at (A3) {};
  \node[draw,circle,inner sep=1.4pt,fill] at (A4) {};
  \node[draw,circle,inner sep=1.4pt,fill] at (A5) {};
  \node[draw,circle,inner sep=1.4pt,fill] at (A6) {};
  \node[draw,circle,inner sep=1.4pt,fill] at (A7) {};
  \node[draw,circle,inner sep=1.4pt,fill] at (A8) {};
\end{scope}

\begin{scope}[shift={(12,3)},scale=1.2]
  \coordinate (A1) at (1,0);
  \coordinate (A2) at (0.707,0.707);
  \coordinate (A3) at (0,1);
  \coordinate (A4) at (-0.707,0.707);
  \coordinate (A5) at (-1,0);
  \coordinate (A6) at (-0.707,-0.707);
  \coordinate (A7) at (0,-1);
  \coordinate (A8) at (0.707,-0.707);

  \draw[very thick,dashed] (A1) -- (A2);
  \draw[very thick] (A1) -- (A3);
  \draw[very thick] (A1) -- (A4);
  \draw[very thick] (A1) -- (A5);
  
  \draw[very thick,dashed] (A2) -- (A3);
  \draw[very thick] (A2) -- (A4);
  \draw[very thick] (A2) -- (A5);
  \draw[very thick] (A2) -- (A6);
  \draw[very thick] (A2) -- (A7);
  \draw[very thick] (A2) -- (A8);
  
  \draw[very thick,dashed] (A3) -- (A4);
  \draw[very thick] (A3) -- (A5);
  \draw[very thick] (A3) -- (A6);
  \draw[very thick] (A3) -- (A7);
  \draw[very thick] (A3) -- (A8);
  
  \draw[very thick,dashed] (A4) -- (A5);
  \draw[very thick] (A4) -- (A6);
  \draw[very thick] (A4) -- (A7);
  \draw[very thick] (A4) -- (A8);
  
  \draw[very thick,dashed] (A5) -- (A6);
  \draw[very thick] (A5) -- (A7);
  \draw[very thick] (A5) -- (A8);
  
  \draw[very thick,dashed] (A6) -- (A7);
  
  \draw[very thick,dashed] (A7) -- (A8);
  
  \node[draw,circle,inner sep=1.4pt,fill] at (A1) {};
  \node[draw,circle,inner sep=1.4pt,fill] at (A2) {};
  \node[draw,circle,inner sep=1.4pt,fill] at (A3) {};
  \node[draw,circle,inner sep=1.4pt,fill] at (A4) {};
  \node[draw,circle,inner sep=1.4pt,fill] at (A5) {};
  \node[draw,circle,inner sep=1.4pt,fill] at (A6) {};
  \node[draw,circle,inner sep=1.4pt,fill] at (A7) {};
  \node[draw,circle,inner sep=1.4pt,fill] at (A8) {};
\end{scope}

\begin{scope}[shift={(0,0)},scale=1.2]
  \coordinate (A1) at (1,0);
  \coordinate (A2) at (0.766,0.643);
  \coordinate (A3) at (0.174,0.985);
  \coordinate (A4) at (-0.5,0.866);
  \coordinate (A5) at (-0.940,0.342);
  \coordinate (A6) at (-0.940,-0.342);
  \coordinate (A7) at (-0.5,-0.866);
  \coordinate (A8) at (0.174,-0.985);
  \coordinate (A9) at (0.766,-0.643);

  \draw[very thick,dashed] (A1) -- (A2);
  \draw[very thick] (A1) -- (A3);
  \draw[very thick] (A1) -- (A4);
  \draw[very thick] (A1) -- (A5);
  \draw[very thick] (A1) -- (A7);
  \draw[very thick] (A1) -- (A8);
  
  \draw[very thick,dashed] (A2) -- (A3);
  \draw[very thick] (A2) -- (A4);
  \draw[very thick] (A2) -- (A6);
  \draw[very thick] (A2) -- (A8);
  \draw[very thick] (A2) -- (A9);
  
  \draw[very thick,dashed] (A3) -- (A5);
  \draw[very thick] (A3) -- (A6);
  \draw[very thick] (A3) -- (A7);
  \draw[very thick] (A3) -- (A9);
  
  \draw[very thick,dashed] (A4) -- (A5);
  \draw[very thick] (A4) -- (A6);
  \draw[very thick] (A4) -- (A7);
  \draw[very thick] (A4) -- (A8);
  
  \draw[very thick,dashed] (A5) -- (A7);
  \draw[very thick] (A5) -- (A8);
  \draw[very thick] (A5) -- (A9);
  
  \draw[very thick,dashed] (A6) -- (A7);
  \draw[very thick] (A6) -- (A8);
  \draw[very thick] (A6) -- (A9);
  
  \draw[very thick,dashed] (A7) -- (A9);

  \draw[very thick,dashed] (A8) -- (A9);
  
  \node[draw,circle,inner sep=1.4pt,fill] at (A1) {};
  \node[draw,circle,inner sep=1.4pt,fill] at (A2) {};
  \node[draw,circle,inner sep=1.4pt,fill] at (A3) {};
  \node[draw,circle,inner sep=1.4pt,fill] at (A4) {};
  \node[draw,circle,inner sep=1.4pt,fill] at (A5) {};
  \node[draw,circle,inner sep=1.4pt,fill] at (A6) {};
  \node[draw,circle,inner sep=1.4pt,fill] at (A7) {};
  \node[draw,circle,inner sep=1.4pt,fill] at (A8) {};
  \node[draw,circle,inner sep=1.4pt,fill] at (A9) {};
\end{scope}

\begin{scope}[shift={(3,0)},scale=1.2]
  \coordinate (A1) at (1,0);
  \coordinate (A2) at (0.766,0.643);
  \coordinate (A3) at (0.174,0.985);
  \coordinate (A4) at (-0.5,0.866);
  \coordinate (A5) at (-0.940,0.342);
  \coordinate (A6) at (-0.940,-0.342);
  \coordinate (A7) at (-0.5,-0.866);
  \coordinate (A8) at (0.174,-0.985);
  \coordinate (A9) at (0.766,-0.643);

  \draw[very thick,dashed] (A1) -- (A2);
  \draw[very thick] (A1) -- (A3);
  \draw[very thick] (A1) -- (A4);
  \draw[very thick] (A1) -- (A5);
  \draw[very thick] (A1) -- (A7);
  \draw[very thick] (A1) -- (A8);
  
  \draw[very thick,dashed] (A2) -- (A3);
  \draw[very thick] (A2) -- (A4);
  \draw[very thick] (A2) -- (A6);
  \draw[very thick] (A2) -- (A8);
  \draw[very thick] (A2) -- (A9);
  
  \draw[very thick,dashed] (A3) -- (A5);
  \draw[very thick] (A3) -- (A6);
  \draw[very thick] (A3) -- (A7);
  \draw[very thick] (A3) -- (A9);
  
  \draw[very thick,dashed] (A4) -- (A5);
  \draw[very thick] (A4) -- (A6);
  \draw[very thick] (A4) -- (A7);
  \draw[very thick] (A4) -- (A9);
  
  \draw[very thick,dashed] (A5) -- (A6);
  \draw[very thick] (A5) -- (A8);
  \draw[very thick] (A5) -- (A9);
  
  \draw[very thick,dashed] (A6) -- (A7);
  \draw[very thick] (A6) -- (A8);

  \draw[very thick,dashed] (A7) -- (A8);
  \draw[very thick] (A7) -- (A9);

  \draw[very thick,dashed] (A8) -- (A9);
  
  \node[draw,circle,inner sep=1.4pt,fill] at (A1) {};
  \node[draw,circle,inner sep=1.4pt,fill] at (A2) {};
  \node[draw,circle,inner sep=1.4pt,fill] at (A3) {};
  \node[draw,circle,inner sep=1.4pt,fill] at (A4) {};
  \node[draw,circle,inner sep=1.4pt,fill] at (A5) {};
  \node[draw,circle,inner sep=1.4pt,fill] at (A6) {};
  \node[draw,circle,inner sep=1.4pt,fill] at (A7) {};
  \node[draw,circle,inner sep=1.4pt,fill] at (A8) {};
  \node[draw,circle,inner sep=1.4pt,fill] at (A9) {};
\end{scope}

\begin{scope}[shift={(6,0)},scale=1.2]
  \coordinate (A1) at (1,0);
  \coordinate (A2) at (0.766,0.643);
  \coordinate (A3) at (0.174,0.985);
  \coordinate (A4) at (-0.5,0.866);
  \coordinate (A5) at (-0.940,0.342);
  \coordinate (A6) at (-0.940,-0.342);
  \coordinate (A7) at (-0.5,-0.866);
  \coordinate (A8) at (0.174,-0.985);
  \coordinate (A9) at (0.766,-0.643);

  \draw[very thick,dashed] (A1) -- (A2);
  \draw[very thick] (A1) -- (A3);
  \draw[very thick] (A1) -- (A4);
  \draw[very thick] (A1) -- (A5);
  \draw[very thick] (A1) -- (A8);
  
  \draw[very thick,dashed] (A2) -- (A3);
  \draw[very thick] (A2) -- (A4);
  \draw[very thick] (A2) -- (A6);
  \draw[very thick,dashed] (A2) -- (A8);
  
  \draw[very thick,dashed] (A3) -- (A4);
  \draw[very thick] (A3) -- (A5);
  \draw[very thick] (A3) -- (A6);
  \draw[very thick] (A3) -- (A7);
  \draw[very thick] (A3) -- (A8);
  \draw[very thick] (A3) -- (A9);
  
  \draw[very thick,dashed] (A4) -- (A5);
  \draw[very thick] (A4) -- (A6);
  \draw[very thick] (A4) -- (A7);
  \draw[very thick] (A4) -- (A8);
  \draw[very thick] (A4) -- (A9);

  \draw[very thick] (A5) -- (A8);
  \draw[very thick] (A5) -- (A9);
  
  \draw[very thick,dashed] (A6) -- (A7);
  \draw[very thick] (A6) -- (A8);

  \draw[very thick,dashed] (A7) -- (A8);
  \draw[very thick] (A7) -- (A9);

  \draw[very thick,dashed] (A8) -- (A9);
  
  \node[draw,circle,inner sep=1.4pt,fill] at (A1) {};
  \node[draw,circle,inner sep=1.4pt,fill] at (A2) {};
  \node[draw,circle,inner sep=1.4pt,fill] at (A3) {};
  \node[draw,circle,inner sep=1.4pt,fill] at (A4) {};
  \node[draw,circle,inner sep=1.4pt,fill] at (A5) {};
  \node[draw,circle,inner sep=1.4pt,fill] at (A6) {};
  \node[draw,circle,inner sep=1.4pt,fill] at (A7) {};
  \node[draw,circle,inner sep=1.4pt,fill] at (A8) {};
  \node[draw,circle,inner sep=1.4pt,fill] at (A9) {};
\end{scope}

\begin{scope}[shift={(9,0)},scale=1.2]
  \coordinate (A1) at (1,0);
  \coordinate (A2) at (0.766,0.643);
  \coordinate (A3) at (0.174,0.985);
  \coordinate (A4) at (-0.5,0.866);
  \coordinate (A5) at (-0.940,0.342);
  \coordinate (A6) at (-0.940,-0.342);
  \coordinate (A7) at (-0.5,-0.866);
  \coordinate (A8) at (0.174,-0.985);
  \coordinate (A9) at (0.766,-0.643);

  \draw[very thick,dashed] (A1) -- (A2);
  \draw[very thick] (A1) -- (A3);
  \draw[very thick] (A1) -- (A4);
  \draw[very thick] (A1) -- (A5);
  \draw[very thick] (A1) -- (A8);
  
  \draw[very thick,dashed] (A2) -- (A3);
  \draw[very thick] (A2) -- (A4);
  \draw[very thick] (A2) -- (A6);
  \draw[very thick] (A2) -- (A7);
  \draw[very thick] (A2) -- (A8);
  \draw[very thick] (A2) -- (A9);
  
  \draw[very thick,dashed] (A3) -- (A4);
  \draw[very thick] (A3) -- (A5);
  \draw[very thick] (A3) -- (A7);
  \draw[very thick] (A3) -- (A9);
  
  \draw[very thick,dashed] (A4) -- (A5);
  \draw[very thick] (A4) -- (A6);
  \draw[very thick] (A4) -- (A7);
  \draw[very thick] (A4) -- (A8);
  \draw[very thick] (A4) -- (A9);

  \draw[very thick,dashed] (A5) -- (A8);
  \draw[very thick] (A5) -- (A9);
  
  \draw[very thick,dashed] (A6) -- (A7);
  \draw[very thick] (A6) -- (A8);
  \draw[very thick] (A6) -- (A9);

  \draw[very thick,dashed] (A7) -- (A9);

  \draw[very thick,dashed] (A8) -- (A9);
  
  \node[draw,circle,inner sep=1.4pt,fill] at (A1) {};
  \node[draw,circle,inner sep=1.4pt,fill] at (A2) {};
  \node[draw,circle,inner sep=1.4pt,fill] at (A3) {};
  \node[draw,circle,inner sep=1.4pt,fill] at (A4) {};
  \node[draw,circle,inner sep=1.4pt,fill] at (A5) {};
  \node[draw,circle,inner sep=1.4pt,fill] at (A6) {};
  \node[draw,circle,inner sep=1.4pt,fill] at (A7) {};
  \node[draw,circle,inner sep=1.4pt,fill] at (A8) {};
  \node[draw,circle,inner sep=1.4pt,fill] at (A9) {};
\end{scope}

\begin{scope}[shift={(12,0)},scale=1.2]
  \coordinate (A1) at (1,0);
  \coordinate (A2) at (0.766,0.643);
  \coordinate (A3) at (0.174,0.985);
  \coordinate (A4) at (-0.5,0.866);
  \coordinate (A5) at (-0.940,0.342);
  \coordinate (A6) at (-0.940,-0.342);
  \coordinate (A7) at (-0.5,-0.866);
  \coordinate (A8) at (0.174,-0.985);
  \coordinate (A9) at (0.766,-0.643);

  \draw[very thick,dashed] (A1) -- (A2);
  \draw[very thick] (A1) -- (A3);
  \draw[very thick] (A1) -- (A4);
  \draw[very thick] (A1) -- (A5);
  \draw[very thick] (A1) -- (A8);
  
  \draw[very thick,dashed] (A2) -- (A3);
  \draw[very thick] (A2) -- (A4);
  \draw[very thick] (A2) -- (A6);
  \draw[very thick] (A2) -- (A8);
  \draw[very thick] (A2) -- (A9);
  
  \draw[very thick,dashed] (A3) -- (A4);
  \draw[very thick] (A3) -- (A5);
  \draw[very thick] (A3) -- (A6);
  \draw[very thick] (A3) -- (A7);
  \draw[very thick] (A3) -- (A9);
  
  \draw[very thick,dashed] (A4) -- (A5);
  \draw[very thick] (A4) -- (A6);
  \draw[very thick] (A4) -- (A7);
  \draw[very thick] (A4) -- (A8);
  \draw[very thick] (A4) -- (A9);

  \draw[very thick,dashed] (A5) -- (A8);
  \draw[very thick] (A5) -- (A9);
  
  \draw[very thick,dashed] (A6) -- (A7);
  \draw[very thick] (A6) -- (A8);

  \draw[very thick,dashed] (A7) -- (A8);
  \draw[very thick] (A7) -- (A9);

  \draw[very thick,dashed] (A8) -- (A9);
  
  \node[draw,circle,inner sep=1.4pt,fill] at (A1) {};
  \node[draw,circle,inner sep=1.4pt,fill] at (A2) {};
  \node[draw,circle,inner sep=1.4pt,fill] at (A3) {};
  \node[draw,circle,inner sep=1.4pt,fill] at (A4) {};
  \node[draw,circle,inner sep=1.4pt,fill] at (A5) {};
  \node[draw,circle,inner sep=1.4pt,fill] at (A6) {};
  \node[draw,circle,inner sep=1.4pt,fill] at (A7) {};
  \node[draw,circle,inner sep=1.4pt,fill] at (A8) {};
  \node[draw,circle,inner sep=1.4pt,fill] at (A9) {};
\end{scope}

\begin{scope}[shift={(0,-3)},scale=1.2]
  \coordinate (A1) at (1,0);
  \coordinate (A2) at (0.766,0.643);
  \coordinate (A3) at (0.174,0.985);
  \coordinate (A4) at (-0.5,0.866);
  \coordinate (A5) at (-0.940,0.342);
  \coordinate (A6) at (-0.940,-0.342);
  \coordinate (A7) at (-0.5,-0.866);
  \coordinate (A8) at (0.174,-0.985);
  \coordinate (A9) at (0.766,-0.643);

  \draw[very thick,dashed] (A1) -- (A2);
  \draw[very thick] (A1) -- (A3);
  \draw[very thick] (A1) -- (A4);
  \draw[very thick] (A1) -- (A5);
  \draw[very thick] (A1) -- (A8);
  
  \draw[very thick,dashed] (A2) -- (A3);
  \draw[very thick] (A2) -- (A4);
  \draw[very thick] (A2) -- (A6);
  \draw[very thick] (A2) -- (A8);
  \draw[very thick] (A2) -- (A9);
  
  \draw[very thick,dashed] (A3) -- (A4);
  \draw[very thick] (A3) -- (A5);
  \draw[very thick] (A3) -- (A6);
  \draw[very thick] (A3) -- (A7);
  \draw[very thick] (A3) -- (A9);
  
  \draw[very thick,dashed] (A4) -- (A5);
  \draw[very thick] (A4) -- (A6);
  \draw[very thick] (A4) -- (A7);
  \draw[very thick] (A4) -- (A8);
  
  \draw[very thick,dashed] (A5) -- (A7);
  \draw[very thick] (A5) -- (A8);
  \draw[very thick] (A5) -- (A9);
  
  \draw[very thick,dashed] (A6) -- (A7);
  \draw[very thick] (A6) -- (A8);
  \draw[very thick] (A6) -- (A9);
  
  \draw[very thick,dashed] (A7) -- (A9);

  \draw[very thick,dashed] (A8) -- (A9);
  
  \node[draw,circle,inner sep=1.4pt,fill] at (A1) {};
  \node[draw,circle,inner sep=1.4pt,fill] at (A2) {};
  \node[draw,circle,inner sep=1.4pt,fill] at (A3) {};
  \node[draw,circle,inner sep=1.4pt,fill] at (A4) {};
  \node[draw,circle,inner sep=1.4pt,fill] at (A5) {};
  \node[draw,circle,inner sep=1.4pt,fill] at (A6) {};
  \node[draw,circle,inner sep=1.4pt,fill] at (A7) {};
  \node[draw,circle,inner sep=1.4pt,fill] at (A8) {};
  \node[draw,circle,inner sep=1.4pt,fill] at (A9) {};
\end{scope}

\begin{scope}[shift={(3,-3)},scale=1.2]
  \coordinate (A1) at (1,0);
  \coordinate (A2) at (0.766,0.643);
  \coordinate (A3) at (0.174,0.985);
  \coordinate (A4) at (-0.5,0.866);
  \coordinate (A5) at (-0.940,0.342);
  \coordinate (A6) at (-0.940,-0.342);
  \coordinate (A7) at (-0.5,-0.866);
  \coordinate (A8) at (0.174,-0.985);
  \coordinate (A9) at (0.766,-0.643);

  \draw[very thick,dashed] (A1) -- (A2);
  \draw[very thick] (A1) -- (A3);
  \draw[very thick] (A1) -- (A4);
  \draw[very thick] (A1) -- (A5);
  
  \draw[very thick,dashed] (A2) -- (A3);
  \draw[very thick] (A2) -- (A4);
  \draw[very thick] (A2) -- (A5);
  \draw[very thick] (A2) -- (A6);
  \draw[very thick] (A2) -- (A8);
  \draw[very thick] (A2) -- (A9);
  
  \draw[very thick,dashed] (A3) -- (A4);
  \draw[very thick] (A3) -- (A5);
  \draw[very thick] (A3) -- (A6);
  \draw[very thick] (A3) -- (A7);
  \draw[very thick] (A3) -- (A9);
  
  \draw[very thick,dashed] (A4) -- (A5);
  \draw[very thick] (A4) -- (A6);
  \draw[very thick] (A4) -- (A7);
  \draw[very thick] (A4) -- (A9);
  
  \draw[very thick,dashed] (A5) -- (A7);
  \draw[very thick] (A5) -- (A8);
  \draw[very thick] (A5) -- (A9);
  
  \draw[very thick,dashed] (A6) -- (A7);
  \draw[very thick] (A6) -- (A8);
  \draw[very thick] (A6) -- (A9);
  
  \draw[very thick,dashed] (A7) -- (A9);

  \draw[very thick,dashed] (A8) -- (A9);
  
  \node[draw,circle,inner sep=1.4pt,fill] at (A1) {};
  \node[draw,circle,inner sep=1.4pt,fill] at (A2) {};
  \node[draw,circle,inner sep=1.4pt,fill] at (A3) {};
  \node[draw,circle,inner sep=1.4pt,fill] at (A4) {};
  \node[draw,circle,inner sep=1.4pt,fill] at (A5) {};
  \node[draw,circle,inner sep=1.4pt,fill] at (A6) {};
  \node[draw,circle,inner sep=1.4pt,fill] at (A7) {};
  \node[draw,circle,inner sep=1.4pt,fill] at (A8) {};
  \node[draw,circle,inner sep=1.4pt,fill] at (A9) {};
\end{scope}

\begin{scope}[shift={(6,-3)},scale=1.2]
  \coordinate (A1) at (1,0);
  \coordinate (A2) at (0.766,0.643);
  \coordinate (A3) at (0.174,0.985);
  \coordinate (A4) at (-0.5,0.866);
  \coordinate (A5) at (-0.940,0.342);
  \coordinate (A6) at (-0.940,-0.342);
  \coordinate (A7) at (-0.5,-0.866);
  \coordinate (A8) at (0.174,-0.985);
  \coordinate (A9) at (0.766,-0.643);

  \draw[very thick,dashed] (A1) -- (A2);
  \draw[very thick] (A1) -- (A3);
  \draw[very thick] (A1) -- (A4);
  \draw[very thick] (A1) -- (A5);
  
  \draw[very thick,dashed] (A2) -- (A3);
  \draw[very thick] (A2) -- (A4);
  \draw[very thick] (A2) -- (A5);
  \draw[very thick] (A2) -- (A6);
  \draw[very thick] (A2) -- (A7);
  \draw[very thick] (A2) -- (A8);
  
  \draw[very thick,dashed] (A3) -- (A4);
  \draw[very thick] (A3) -- (A5);
  \draw[very thick] (A3) -- (A7);
  \draw[very thick] (A3) -- (A8);
  \draw[very thick] (A3) -- (A9);
  
  \draw[very thick,dashed] (A4) -- (A5);
  \draw[very thick] (A4) -- (A6);
  \draw[very thick] (A4) -- (A7);
  \draw[very thick] (A4) -- (A9);
  
  \draw[very thick,dashed] (A5) -- (A6);
  \draw[very thick] (A5) -- (A8);
  \draw[very thick] (A5) -- (A9);
  
  \draw[very thick,dashed] (A6) -- (A7);
  \draw[very thick] (A6) -- (A8);
  \draw[very thick] (A6) -- (A9);
  
  \draw[very thick,dashed] (A7) -- (A9);

  \draw[very thick,dashed] (A8) -- (A9);
  
  \node[draw,circle,inner sep=1.4pt,fill] at (A1) {};
  \node[draw,circle,inner sep=1.4pt,fill] at (A2) {};
  \node[draw,circle,inner sep=1.4pt,fill] at (A3) {};
  \node[draw,circle,inner sep=1.4pt,fill] at (A4) {};
  \node[draw,circle,inner sep=1.4pt,fill] at (A5) {};
  \node[draw,circle,inner sep=1.4pt,fill] at (A6) {};
  \node[draw,circle,inner sep=1.4pt,fill] at (A7) {};
  \node[draw,circle,inner sep=1.4pt,fill] at (A8) {};
  \node[draw,circle,inner sep=1.4pt,fill] at (A9) {};
\end{scope}

\begin{scope}[shift={(9,-3)},scale=1.2]
  \coordinate (A1) at (1,0);
  \coordinate (A2) at (0.766,0.643);
  \coordinate (A3) at (0.174,0.985);
  \coordinate (A4) at (-0.5,0.866);
  \coordinate (A5) at (-0.940,0.342);
  \coordinate (A6) at (-0.940,-0.342);
  \coordinate (A7) at (-0.5,-0.866);
  \coordinate (A8) at (0.174,-0.985);
  \coordinate (A9) at (0.766,-0.643);

  \draw[very thick,dashed] (A1) -- (A2);
  \draw[very thick] (A1) -- (A3);
  \draw[very thick] (A1) -- (A4);
  \draw[very thick] (A1) -- (A5);
  
  \draw[very thick,dashed] (A2) -- (A3);
  \draw[very thick] (A2) -- (A4);
  \draw[very thick] (A2) -- (A5);
  \draw[very thick] (A2) -- (A6);
  \draw[very thick] (A2) -- (A8);
  
  \draw[very thick,dashed] (A3) -- (A4);
  \draw[very thick] (A3) -- (A5);
  \draw[very thick] (A3) -- (A6);
  \draw[very thick] (A3) -- (A7);
  \draw[very thick] (A3) -- (A8);
  \draw[very thick] (A3) -- (A9);
  
  \draw[very thick,dashed] (A4) -- (A5);
  \draw[very thick] (A4) -- (A6);
  \draw[very thick] (A4) -- (A7);
  \draw[very thick] (A4) -- (A9);
  
  \draw[very thick,dashed] (A5) -- (A6);
  \draw[very thick] (A5) -- (A8);
  \draw[very thick] (A5) -- (A9);
  
  \draw[very thick,dashed] (A6) -- (A7);
  \draw[very thick] (A6) -- (A8);
  \draw[very thick] (A6) -- (A9);
  
  \draw[very thick,dashed] (A7) -- (A9);

  \draw[very thick,dashed] (A8) -- (A9);
  
  \node[draw,circle,inner sep=1.4pt,fill] at (A1) {};
  \node[draw,circle,inner sep=1.4pt,fill] at (A2) {};
  \node[draw,circle,inner sep=1.4pt,fill] at (A3) {};
  \node[draw,circle,inner sep=1.4pt,fill] at (A4) {};
  \node[draw,circle,inner sep=1.4pt,fill] at (A5) {};
  \node[draw,circle,inner sep=1.4pt,fill] at (A6) {};
  \node[draw,circle,inner sep=1.4pt,fill] at (A7) {};
  \node[draw,circle,inner sep=1.4pt,fill] at (A8) {};
  \node[draw,circle,inner sep=1.4pt,fill] at (A9) {};
\end{scope}

\begin{scope}[shift={(12,-3)},scale=1.2]
  \coordinate (A1) at (1,0);
  \coordinate (A2) at (0.766,0.643);
  \coordinate (A3) at (0.174,0.985);
  \coordinate (A4) at (-0.5,0.866);
  \coordinate (A5) at (-0.940,0.342);
  \coordinate (A6) at (-0.940,-0.342);
  \coordinate (A7) at (-0.5,-0.866);
  \coordinate (A8) at (0.174,-0.985);
  \coordinate (A9) at (0.766,-0.643);

  \draw[very thick,dashed] (A1) -- (A2);
  \draw[very thick] (A1) -- (A3);
  \draw[very thick] (A1) -- (A4);
  \draw[very thick] (A1) -- (A5);
  
  \draw[very thick,dashed] (A2) -- (A3);
  \draw[very thick] (A2) -- (A4);
  \draw[very thick] (A2) -- (A5);
  \draw[very thick] (A2) -- (A6);
  \draw[very thick] (A2) -- (A8);
  
  \draw[very thick,dashed] (A3) -- (A4);
  \draw[very thick] (A3) -- (A5);
  \draw[very thick] (A3) -- (A6);
  \draw[very thick] (A3) -- (A7);
  \draw[very thick] (A3) -- (A8);
  \draw[very thick] (A3) -- (A9);
  
  \draw[very thick,dashed] (A4) -- (A5);
  \draw[very thick] (A4) -- (A6);
  \draw[very thick] (A4) -- (A7);
  \draw[very thick] (A4) -- (A9);
  
  \draw[very thick,dashed] (A5) -- (A6);
  \draw[very thick] (A5) -- (A8);
  \draw[very thick] (A5) -- (A9);
  
  \draw[very thick,dashed] (A6) -- (A7);
  \draw[very thick] (A6) -- (A8);
  
  \draw[very thick,dashed] (A7) -- (A8);
  \draw[very thick] (A7) -- (A9);

  \draw[very thick,dashed] (A8) -- (A9);
  
  \node[draw,circle,inner sep=1.4pt,fill] at (A1) {};
  \node[draw,circle,inner sep=1.4pt,fill] at (A2) {};
  \node[draw,circle,inner sep=1.4pt,fill] at (A3) {};
  \node[draw,circle,inner sep=1.4pt,fill] at (A4) {};
  \node[draw,circle,inner sep=1.4pt,fill] at (A5) {};
  \node[draw,circle,inner sep=1.4pt,fill] at (A6) {};
  \node[draw,circle,inner sep=1.4pt,fill] at (A7) {};
  \node[draw,circle,inner sep=1.4pt,fill] at (A8) {};
  \node[draw,circle,inner sep=1.4pt,fill] at (A9) {};
\end{scope}

\begin{scope}[shift={(0,-6)},scale=1.2]
  \coordinate (A1) at (1,0);
  \coordinate (A2) at (0.766,0.643);
  \coordinate (A3) at (0.174,0.985);
  \coordinate (A4) at (-0.5,0.866);
  \coordinate (A5) at (-0.940,0.342);
  \coordinate (A6) at (-0.940,-0.342);
  \coordinate (A7) at (-0.5,-0.866);
  \coordinate (A8) at (0.174,-0.985);
  \coordinate (A9) at (0.766,-0.643);

  \draw[very thick,dashed] (A1) -- (A2);
  \draw[very thick] (A1) -- (A3);
  \draw[very thick] (A1) -- (A4);
  \draw[very thick] (A1) -- (A5);
  
  \draw[very thick,dashed] (A2) -- (A3);
  \draw[very thick] (A2) -- (A4);
  \draw[very thick] (A2) -- (A5);
  \draw[very thick] (A2) -- (A6);
  \draw[very thick] (A2) -- (A8);
  
  \draw[very thick,dashed] (A3) -- (A4);
  \draw[very thick] (A3) -- (A5);
  \draw[very thick] (A3) -- (A6);
  \draw[very thick] (A3) -- (A7);
  \draw[very thick] (A3) -- (A8);
  \draw[very thick] (A3) -- (A9);
  
  \draw[very thick,dashed] (A4) -- (A5);
  \draw[very thick] (A4) -- (A6);
  \draw[very thick] (A4) -- (A7);
  
  \draw[very thick,dashed] (A5) -- (A6);
  \draw[very thick] (A5) -- (A7);
  \draw[very thick] (A5) -- (A8);
  \draw[very thick] (A5) -- (A9);
  
  \draw[very thick,dashed] (A6) -- (A7);
  \draw[very thick] (A6) -- (A8);
  
  \draw[very thick,dashed] (A7) -- (A8);
  \draw[very thick] (A7) -- (A9);

  \draw[very thick,dashed] (A8) -- (A9);
  
  \node[draw,circle,inner sep=1.4pt,fill] at (A1) {};
  \node[draw,circle,inner sep=1.4pt,fill] at (A2) {};
  \node[draw,circle,inner sep=1.4pt,fill] at (A3) {};
  \node[draw,circle,inner sep=1.4pt,fill] at (A4) {};
  \node[draw,circle,inner sep=1.4pt,fill] at (A5) {};
  \node[draw,circle,inner sep=1.4pt,fill] at (A6) {};
  \node[draw,circle,inner sep=1.4pt,fill] at (A7) {};
  \node[draw,circle,inner sep=1.4pt,fill] at (A8) {};
  \node[draw,circle,inner sep=1.4pt,fill] at (A9) {};
\end{scope}

\begin{scope}[shift={(3,-6)},scale=1.2]
  \coordinate (A1) at (1,0);
  \coordinate (A2) at (0.766,0.643);
  \coordinate (A3) at (0.174,0.985);
  \coordinate (A4) at (-0.5,0.866);
  \coordinate (A5) at (-0.940,0.342);
  \coordinate (A6) at (-0.940,-0.342);
  \coordinate (A7) at (-0.5,-0.866);
  \coordinate (A8) at (0.174,-0.985);
  \coordinate (A9) at (0.766,-0.643);

  \draw[very thick,dashed] (A1) -- (A2);
  \draw[very thick] (A1) -- (A3);
  \draw[very thick] (A1) -- (A4);
  \draw[very thick] (A1) -- (A5);
  
  \draw[very thick,dashed] (A2) -- (A3);
  \draw[very thick] (A2) -- (A4);
  \draw[very thick] (A2) -- (A5);
  \draw[very thick] (A2) -- (A6);
  \draw[very thick] (A2) -- (A8);
  
  \draw[very thick,dashed] (A3) -- (A4);
  \draw[very thick] (A3) -- (A5);
  \draw[very thick] (A3) -- (A6);
  \draw[very thick] (A3) -- (A7);
  \draw[very thick] (A3) -- (A8);
  \draw[very thick] (A3) -- (A9);
  
  \draw[very thick,dashed] (A4) -- (A5);
  \draw[very thick] (A4) -- (A6);
  \draw[very thick] (A4) -- (A7);
  
  \draw[very thick,dashed] (A5) -- (A6);
  \draw[very thick] (A5) -- (A7);
  \draw[very thick] (A5) -- (A8);
  \draw[very thick] (A5) -- (A9);
  
  \draw[very thick,dashed] (A6) -- (A7);
  \draw[very thick] (A6) -- (A8);
  \draw[very thick] (A6) -- (A9);

  \draw[very thick,dashed] (A7) -- (A9);

  \draw[very thick,dashed] (A8) -- (A9);
  
  \node[draw,circle,inner sep=1.4pt,fill] at (A1) {};
  \node[draw,circle,inner sep=1.4pt,fill] at (A2) {};
  \node[draw,circle,inner sep=1.4pt,fill] at (A3) {};
  \node[draw,circle,inner sep=1.4pt,fill] at (A4) {};
  \node[draw,circle,inner sep=1.4pt,fill] at (A5) {};
  \node[draw,circle,inner sep=1.4pt,fill] at (A6) {};
  \node[draw,circle,inner sep=1.4pt,fill] at (A7) {};
  \node[draw,circle,inner sep=1.4pt,fill] at (A8) {};
  \node[draw,circle,inner sep=1.4pt,fill] at (A9) {};
\end{scope}

\begin{scope}[shift={(6,-6)},scale=1.2]
  \coordinate (A1) at (1,0);
  \coordinate (A2) at (0.766,0.643);
  \coordinate (A3) at (0.174,0.985);
  \coordinate (A4) at (-0.5,0.866);
  \coordinate (A5) at (-0.940,0.342);
  \coordinate (A6) at (-0.940,-0.342);
  \coordinate (A7) at (-0.5,-0.866);
  \coordinate (A8) at (0.174,-0.985);
  \coordinate (A9) at (0.766,-0.643);

  \draw[very thick,dashed] (A1) -- (A2);
  \draw[very thick] (A1) -- (A3);
  \draw[very thick] (A1) -- (A4);
  \draw[very thick] (A1) -- (A5);
  
  \draw[very thick,dashed] (A2) -- (A3);
  \draw[very thick] (A2) -- (A4);
  \draw[very thick] (A2) -- (A5);
  \draw[very thick] (A2) -- (A6);
  \draw[very thick] (A2) -- (A7);
  \draw[very thick] (A2) -- (A8);
  
  \draw[very thick,dashed] (A3) -- (A4);
  \draw[very thick] (A3) -- (A5);
  \draw[very thick] (A3) -- (A7);
  \draw[very thick] (A3) -- (A8);
  \draw[very thick] (A3) -- (A9);
  
  \draw[very thick,dashed] (A4) -- (A5);
  \draw[very thick] (A4) -- (A6);
  \draw[very thick] (A4) -- (A7);
  \draw[very thick] (A4) -- (A8);
  \draw[very thick] (A4) -- (A9);
  
  \draw[very thick,dashed] (A5) -- (A6);
  \draw[very thick] (A5) -- (A8);
  \draw[very thick] (A5) -- (A9);
  
  \draw[very thick,dashed] (A6) -- (A7);
  \draw[very thick] (A6) -- (A8);

  \draw[very thick,dashed] (A7) -- (A8);

  \draw[very thick,dashed] (A8) -- (A9);
  
  \node[draw,circle,inner sep=1.4pt,fill] at (A1) {};
  \node[draw,circle,inner sep=1.4pt,fill] at (A2) {};
  \node[draw,circle,inner sep=1.4pt,fill] at (A3) {};
  \node[draw,circle,inner sep=1.4pt,fill] at (A4) {};
  \node[draw,circle,inner sep=1.4pt,fill] at (A5) {};
  \node[draw,circle,inner sep=1.4pt,fill] at (A6) {};
  \node[draw,circle,inner sep=1.4pt,fill] at (A7) {};
  \node[draw,circle,inner sep=1.4pt,fill] at (A8) {};
  \node[draw,circle,inner sep=1.4pt,fill] at (A9) {};
\end{scope}

\begin{scope}[shift={(9,-6)},scale=1.2]
  \coordinate (A1) at (1,0);
  \coordinate (A2) at (0.766,0.643);
  \coordinate (A3) at (0.174,0.985);
  \coordinate (A4) at (-0.5,0.866);
  \coordinate (A5) at (-0.940,0.342);
  \coordinate (A6) at (-0.940,-0.342);
  \coordinate (A7) at (-0.5,-0.866);
  \coordinate (A8) at (0.174,-0.985);
  \coordinate (A9) at (0.766,-0.643);

  \draw[very thick,dashed] (A1) -- (A2);
  \draw[very thick] (A1) -- (A3);
  \draw[very thick] (A1) -- (A4);
  \draw[very thick] (A1) -- (A5);
  
  \draw[very thick,dashed] (A2) -- (A3);
  \draw[very thick] (A2) -- (A4);
  \draw[very thick] (A2) -- (A5);
  \draw[very thick] (A2) -- (A6);
  \draw[very thick] (A2) -- (A7);
  \draw[very thick] (A2) -- (A8);
  
  \draw[very thick,dashed] (A3) -- (A4);
  \draw[very thick] (A3) -- (A5);
  \draw[very thick] (A3) -- (A7);
  \draw[very thick] (A3) -- (A8);
  \draw[very thick] (A3) -- (A9);
  
  \draw[very thick,dashed] (A4) -- (A5);
  \draw[very thick] (A4) -- (A6);
  \draw[very thick] (A4) -- (A7);
  \draw[very thick] (A4) -- (A8);
  \draw[very thick] (A4) -- (A9);
  
  \draw[very thick,dashed] (A5) -- (A6);
  \draw[very thick] (A5) -- (A7);
  \draw[very thick] (A5) -- (A8);
  \draw[very thick] (A5) -- (A9);
  
  \draw[very thick,dashed] (A6) -- (A7);

  \draw[very thick,dashed] (A7) -- (A8);

  \draw[very thick,dashed] (A8) -- (A9);
  
  \node[draw,circle,inner sep=1.4pt,fill] at (A1) {};
  \node[draw,circle,inner sep=1.4pt,fill] at (A2) {};
  \node[draw,circle,inner sep=1.4pt,fill] at (A3) {};
  \node[draw,circle,inner sep=1.4pt,fill] at (A4) {};
  \node[draw,circle,inner sep=1.4pt,fill] at (A5) {};
  \node[draw,circle,inner sep=1.4pt,fill] at (A6) {};
  \node[draw,circle,inner sep=1.4pt,fill] at (A7) {};
  \node[draw,circle,inner sep=1.4pt,fill] at (A8) {};
  \node[draw,circle,inner sep=1.4pt,fill] at (A9) {};
\end{scope}

\begin{scope}[shift={(12,-6)},scale=1.2]
  \coordinate (A1) at (1,0);
  \coordinate (A2) at (0.766,0.643);
  \coordinate (A3) at (0.174,0.985);
  \coordinate (A4) at (-0.5,0.866);
  \coordinate (A5) at (-0.940,0.342);
  \coordinate (A6) at (-0.940,-0.342);
  \coordinate (A7) at (-0.5,-0.866);
  \coordinate (A8) at (0.174,-0.985);
  \coordinate (A9) at (0.766,-0.643);

  \draw[very thick,dashed] (A1) -- (A2);
  \draw[very thick] (A1) -- (A3);
  \draw[very thick] (A1) -- (A4);
  \draw[very thick] (A1) -- (A5);
  
  \draw[very thick,dashed] (A2) -- (A3);
  \draw[very thick] (A2) -- (A4);
  \draw[very thick] (A2) -- (A5);
  \draw[very thick] (A2) -- (A6);
  \draw[very thick] (A2) -- (A8);
  
  \draw[very thick,dashed] (A3) -- (A4);
  \draw[very thick] (A3) -- (A5);
  \draw[very thick] (A3) -- (A6);
  \draw[very thick] (A3) -- (A7);
  \draw[very thick] (A3) -- (A8);
  \draw[very thick] (A3) -- (A9);
  
  \draw[very thick,dashed] (A4) -- (A5);
  \draw[very thick] (A4) -- (A6);
  \draw[very thick] (A4) -- (A7);
  \draw[very thick] (A4) -- (A8);
  \draw[very thick] (A4) -- (A9);
  
  \draw[very thick,dashed] (A5) -- (A6);
  \draw[very thick] (A5) -- (A7);
  \draw[very thick] (A5) -- (A8);
  \draw[very thick] (A5) -- (A9);
  
  \draw[very thick,dashed] (A6) -- (A7);

  \draw[very thick,dashed] (A7) -- (A8);

  \draw[very thick,dashed] (A8) -- (A9);
  
  \node[draw,circle,inner sep=1.4pt,fill] at (A1) {};
  \node[draw,circle,inner sep=1.4pt,fill] at (A2) {};
  \node[draw,circle,inner sep=1.4pt,fill] at (A3) {};
  \node[draw,circle,inner sep=1.4pt,fill] at (A4) {};
  \node[draw,circle,inner sep=1.4pt,fill] at (A5) {};
  \node[draw,circle,inner sep=1.4pt,fill] at (A6) {};
  \node[draw,circle,inner sep=1.4pt,fill] at (A7) {};
  \node[draw,circle,inner sep=1.4pt,fill] at (A8) {};
  \node[draw,circle,inner sep=1.4pt,fill] at (A9) {};
\end{scope}

\begin{scope}[shift={(6,-9)},scale=1.4]
  \coordinate (A1) at (1,0);
  \coordinate (A2) at (0.809,0.588);
  \coordinate (A3) at (0.309,0.951);
  \coordinate (A4) at (-0.309,0.951);
  \coordinate (A5) at (-0.809,0.588);
  \coordinate (A6) at (-1,0);
  \coordinate (A7) at (-0.809,-0.588);
  \coordinate (A8) at (-0.309,-0.951);
  \coordinate (A9) at (0.309,-0.951);
  \coordinate (A10) at (0.809,-0.588);

  \draw[very thick,dashed] (A1) -- (A2);
  \draw[very thick] (A1) -- (A3);
  \draw[very thick] (A1) -- (A4);
  \draw[very thick] (A1) -- (A5);
  
  \draw[very thick,dashed] (A2) -- (A3);
  \draw[very thick] (A2) -- (A4);
  \draw[very thick] (A2) -- (A5);
  \draw[very thick] (A2) -- (A6);
  \draw[very thick] (A2) -- (A7);
  \draw[very thick] (A2) -- (A9);
  \draw[very thick] (A2) -- (A10);
  
  \draw[very thick,dashed] (A3) -- (A4);
  \draw[very thick] (A3) -- (A5);
  \draw[very thick] (A3) -- (A7);
  \draw[very thick] (A3) -- (A8);
  \draw[very thick] (A3) -- (A9);
  \draw[very thick] (A3) -- (A10);
  
  \draw[very thick,dashed] (A4) -- (A5);
  \draw[very thick] (A4) -- (A6);
  \draw[very thick] (A4) -- (A7);
  \draw[very thick] (A4) -- (A8);
  \draw[very thick] (A4) -- (A10);
  
  \draw[very thick,dashed] (A5) -- (A6);
  \draw[very thick] (A5) -- (A8);
  \draw[very thick] (A5) -- (A9);
  \draw[very thick] (A5) -- (A10);
  
  \draw[very thick,dashed] (A6) -- (A10);
  
  \draw[very thick,dashed] (A7) -- (A10);
  
  \draw[very thick,dashed] (A8) -- (A10);
  
  \draw[very thick,dashed] (A9) -- (A10);
  
  \node[draw,circle,inner sep=1.4pt,fill] at (A1) {};
  \node[draw,circle,inner sep=1.4pt,fill] at (A2) {};
  \node[draw,circle,inner sep=1.4pt,fill] at (A3) {};
  \node[draw,circle,inner sep=1.4pt,fill] at (A4) {};
  \node[draw,circle,inner sep=1.4pt,fill] at (A5) {};
  \node[draw,circle,inner sep=1.4pt,fill] at (A6) {};
  \node[draw,circle,inner sep=1.4pt,fill] at (A7) {};
  \node[draw,circle,inner sep=1.4pt,fill] at (A8) {};
  \node[draw,circle,inner sep=1.4pt,fill] at (A9) {};
  \node[draw,circle,inner sep=1.4pt,fill] at (A10) {};
\end{scope}
\end{tikzpicture}
\caption{All irreducible triangulations of a torus with 8 or more vertices decomposed into a spanning tree (dashed edges) and a spanning subgraph that is rigid in $\mathbb{R}^2$;
each of the latter subgraphs are rigid since they can be constructed from $K_3$ by a sequence of 2-dimensional 0-extensions and edge additions.}
\label{fig:torus}
\end{figure}

\newpage

\section{Application of Edmond's algorithm}
\label{sec:alg}

The following is an application of Edmond's algorithm \cite{edmond} that can be used to determine whether a graph is rigid in a cylindrical space $X \oplus_\infty \mathbb{R}$,
where $X$ is a generic space.
See \cite[Algorithm 11.1]{matroidunion} for a description of the general algorithm. 

\begin{algorithm}
\textbf{Input:} Graph $G=(V,E)$, edge set $A \subset E$ such that $(V,A)$ is a maximal independent subgraph of $G$ in $X$, edge set $B \subset E$ such that $(V,B)$ is a maximal cycle-free subgraph of $G$.\\
\textbf{Output:} Maximal independent subgraph $H$ of $G$ in $X \oplus_\infty \mathbb{R}$.
\caption{ Find a maximal subgraph that is independent in $X \oplus_\infty \mathbb{R}$.}
\label{alg1}
\begin{algorithmic}
\State $I \gets A$
\State $J \gets B$
\State Repeat $\gets$ True
\While{Repeat \textbf{is} True}
    \State Repeat $\gets$ False
    \State $D \gets $ directed graph with vertices $V(D) = E$ and directed edges $E(D) = \emptyset$
    \For{$e \in I$, $f \in E \setminus I$}
        \If{$(V, I -e +f)$ is independent in $X$}
            \State $E(D) \gets E(D) \cup \{ (e,f)\}$
        \EndIf
    \EndFor
    \For{$e \in J$, $f \in E \setminus J$ }
        \If{$(V, J -e +f)$ contains no cycles}
            \State $E(D) \gets E(D) \cup \{ (e,f)\}$
        \EndIf
    \EndFor
    \For{$e_0 \in I \cap J$}
        \If{search of $D$ from $e_0$ finds directed path $e_0, \ldots, e_n, e$ with $e \in E \setminus (I \cup J)$}
            \If{$e_n \in I$}
                \State $I \gets I \triangle  \{e_0,\ldots,e_n,e\}$
                \State $J \gets J \triangle   \{e_0,\ldots,e_n\}$
            \ElsIf{$e_n \in J$}
                \State $I \gets I \triangle   \{e_0,\ldots,e_n\}$
                \State $J \gets J \triangle    \{e_0,\ldots,e_n,e\}$
            \EndIf
            \State Repeat $\gets$ True
            \State \textbf{break for loop}
        \EndIf
    \EndFor
\EndWhile
\State \Return $H = (V, I \cup J)$
\end{algorithmic}
\end{algorithm}

\end{appendices}


\begin{thebibliography}{1}

\bibitem{Barnette}
D.~Barnette,
{\em Generating the triangulations of the projective plane},
Journal of Combinatorial Theory, Series B,
33(3) (1982),
pp. 222--230.


\bibitem{Bry}
T.~Brylawski,
{\em Constructions},
In: N.~White (Ed.), {\em Theory of Matroids},
Encyclopedia of Mathematics and its Applications,
Cambridge University Press (1986),
pp.~127--223.



\bibitem{packing}
J.~Cheriyana, O.~Durand de Gevigney, Z.~Szigeti,
{\em Packing of rigid spanning subgraphs and spanning trees},
Journal of Combinatorial Theory, Series B,
105 (2014),
pp.~17--25.

\bibitem{clikit} K.~Clinch, D.~Kitson, 
{\em Constructing isostatic frameworks for the $\ell^1$ and $\ell^\infty$-plane}, 
Electronic Journal of Combinatorics 
27(2) (2020) \#P2.49.

\bibitem{CAPR} 
C.~Cros, P.~O.~Amblard, C.~Prieur, J.~F.~Da Rocha.
{\em Conic frameworks infinitesimal rigidity},
preprint (2022),
arXiv:2207.03310. 

\bibitem{CruKasKitSch}
J. Cruickshank, E. Kastis, D. Kitson, B. Schulze,
{\em Braced triangulations and rigidity},
preprint (2021),
arXiv:2107.03829.

\bibitem{dew1} 
S. Dewar, 
{\em Equivalence of continuous, local and infinitesimal rigidity in normed spaces},
Discrete \& Computational Geometry,
65 (2021),
pp.~655--679. 

\bibitem{dew2} 
S. Dewar, 
{\em Infinitesimal rigidity in normed planes}, 
Siam Journal on Discrete Mathematics,
34(2) (2020),
pp. 1205--1231.


\bibitem{DewKitNix} 
S. Dewar, D. Kitson, A. Nixon,
{\em Which graphs are rigid in $\ell_p^d$?},
Journal of Global Optimization,
83 (2022),
pp.~49--71. 

\bibitem{edmond}
J.~Edmonds,
{\em Minimum partition of a matroid into independent subsets},
Journal of Research of the National Bureau of Standards, 69B (1965), pp.~67--72.

\bibitem{frankszego03}
A.~Frank, L.~Szeg\"{o},
{\em Constructive characterizations for packing and covering with trees},
Discrete Applied Mathematics,
131(2) (2003),
pp.~347--371.

\bibitem{GrSS93} 
J.~Graver, B.~Servatius, H.~Servatius, 
{\em Combinatorial Rigidity}, 
Graduate Studies in Mathematics,  Volume 2, American Mathematical Society (1993).

\bibitem{gu}
X.~Gu,
{\em Spanning rigid subgraph packing and sparse subgraph covering},
Siam Journal on Discrete Mathematics, 
32(2) (2018),
pp.~1305--1313.

\bibitem{polyhedra} 
D. Kitson, 
{\em Finite and infinitesimal rigidity with polyhedral norms}, 
Discrete \& Computational Geometry,
54(2) (2015),
pp.~390--411.
%
\bibitem{matrixnorm}
D. Kitson, R. H. Levene,
{\em Graph rigidity for unitarily invariant matrix norms}, 
Journal of Mathematical Analysis and Applications,
491(2) (2020),
124353.

\bibitem{kit-pow}
D.~Kitson, S.~C.~Power,
{\em Infinitesimal rigidity for non-Euclidean bar-joint frameworks}, 
Bulletin of the London Mathematical Society,
46(4) (2014),
pp.~685--697.

\bibitem{kit-pow-inf}
D.~Kitson, S.~C.~Power,
{\em The rigidity of infinite graphs}, 
Discrete \& Computational Geometry,
60 (2018),
pp.~531--557.

\bibitem{kit-sch}
D. Kitson, B.~Schulze,
{\em Maxwell–Laman counts for bar-joint frameworks in normed spaces}, 
Linear Algebra and its Applications,
481 (2015),
pp.~313--329.

\bibitem{kundu} 
S.~Kundu,
{\em Bounds on the number of disjoint spanning trees},
Journal of Combinatorial Theory, Series B, 
17 (1974), 
pp.~199--203.

\bibitem{laman}
G.~Laman,
{\em On graphs and the rigidity of plane skeletal structures},
Journal of Engineering Mathematics,
4 (1927),
pp.~331--340.

\bibitem{lav}
S. A. Lavrenchenko, 
{\em Irreducible triangulations of the torus},
Journal of Soviet Mathematics,
51 (1990),
pp.~2537--2543.


\bibitem{LeeStreinu}
A.~Lee, I.~Streinu,
{\em Pebble game algorithms and sparse graphs},
Discrete Mathematics,
308(8) (2008),
pp.~1425--1437.

\bibitem{lovasz}
L. Lov\'{a}sz \and Y. Yemini,
{\em On generic rigidity in the plane},
SIAM Journal on Algebraic Discrete Methods,
3(1) (1982),
pp.~91--98.


\bibitem{MS43}
D.~Montgomery, H.~Samelson,
{\em Transformation groups of spheres},
Annals of Mathematics, 
44(3) (1943), 
pp. 454--470.

\bibitem{matroidunion}
H.~Narayanan,
{\em Submodular Functions and Electrical Networks},
Annals of Discrete Mathematics 54, Elsevier, North Holland,
(1997).


\bibitem{nashwill}
C.~St.J.~A.~Nash-Williams, 
{\em Edge-disjoint spanning trees of finite graphs},
Journal of the London Mathematical Society,
s1-36(1) (1961),
pp.~445--450.


\bibitem{pollaczek27}
H.~Pollaczek‐Geiringer,
{\em Über die Gliederung ebener Fachwerke},
Zeitschrift für Angewandte Mathematik und Mechanik, 
7(1) (1927),
pp.~58--72.


\bibitem{minkowski}
A.~C.~Thompson,
{\em Minkowski geometry},
Encyclopedia of Mathematics and its Applications, 
Cambridge University Press (1996).



\end{thebibliography}
\end{document}